\documentclass[reqno]{amsart}

\usepackage{amssymb,amsfonts,amscd}
\usepackage[all,arc]{xy}
\usepackage[margin=1.2in]{geometry} 
\usepackage{enumerate}
\usepackage{mathrsfs}
\usepackage{hyperref}
\usepackage[usenames, dvipsnames]{color}
\usepackage{graphicx}
\usepackage{enumitem} 
\usepackage{tikz}
\graphicspath{ {TexImages/} }

\newtheorem{thm}{Theorem}[section]
\newtheorem{cor}[thm]{Corollary}
\newtheorem{prop}[thm]{Proposition}
\newtheorem{lem}[thm]{Lemma}

\newtheorem{prob}[thm]{Problem}
\theoremstyle{definition}
\newtheorem{defn}[thm]{Definition}

\theoremstyle{remark}

\numberwithin{equation}{section}

\setenumerate[1]{label=\thesection.\arabic*.} 
\setenumerate[2]{label*=\arabic*.} 
\setlist[enumerate]{itemsep=2ex, topsep=2ex} 
\setlist[itemize]{itemsep=2ex, topsep=2ex}

\newcommand{\al}{\alpha}
\newcommand{\be}{\beta}

\renewcommand{\l}{\left}
\renewcommand{\r}{\right}
\newcommand{\Om}{\Omega}
\renewcommand{\Pr}{\mathbb{P}}
\newcommand{\half}{\frac{1}{2}}
\newcommand{\quart}{\frac{1}{4}}
\newcommand{\tr}[1]{\textrm{#1}}
\newcommand{\om}{\omega}
\newcommand{\rec}[1]{\frac{1}{#1}}
\newcommand{\f}[2]{\frac{#1}{#2}}

\newcommand{\floor}[1]{\lfloor #1\rfloor}
\newcommand{\gam}{\gamma}

\renewcommand{\b}[1]{\mathbf{#1}}
\newcommand{\E}{\mathbb{E}}
\newcommand{\simac}{\sim_{\mathrm{ac}}}

\newcommand{\Var}{\mathrm{Var}}

\newcommand{\red}{\color{red}}
\newcommand{\blue}{\color{blue}}

\hypersetup{
	colorlinks,
	citecolor=blue,
	filecolor=blue,
	linkcolor=blue,
	urlcolor=blue,
	linktocpage
}

\begin{document}

\title{Random graphs induced by Catalan pairs}

\author{Dani\"el Kroes and Sam Spiro}
\address
{Department of Mathematics \newline \indent
University of California, San Diego \newline \indent
La Jolla, CA, 92093-0112, USA}
\email{(dkroes,sspiro)@ucsd.edu}

\begin{abstract}
We consider Catalan-pair graphs, a family of graphs that can be viewed as representing certain interactions between pairs of objects which are enumerated by the Catalan numbers. In this paper we study random Catalan-pair graphs and deduce various properties of these random graphs. In particular, we asymptotically determine the expected number of edges and isolated vertices, and more generally we determine the expected number of (induced) subgraphs isomorphic to a given connected graph.
\end{abstract}

\maketitle

\section{Introduction} \label{sec-introduction}
A large body of work has been devoted to studying the Catalan numbers, as well as the many objects that these numbers enumerate.  Such objects include polygon triangulations, binary trees, plane trees, and Dyck paths. For a thorough treatment of Catalan numbers and their history, we refer the reader to \cite{PakCat} and \cite{StanleyCat}.  In this paper we are interested in examining pairs of objects enumerated by the Catalan numbers, as opposed to looking at just a single such object.  In particular, we will be interested in studying how the objects in these pairs interact with one another, and we will represent this interaction as a graph.

To this end, recall that the Catalan numbers count the number of ways one can place $n$ non-intersecting semi-circular arcs on $2n$ given collinear points. We will refer to such a placement of arcs as a \emph{Catalan-arc matching (of size $n$)}. For example, below one can see all $5$ Catalan-arc matchings of size $3$.
\[
\begin{tikzpicture}[scale=0.3]
\node at (0,0) {$\bullet$};
\node at (1,0) {$\bullet$};
\node at (2,0) {$\bullet$};
\node at (3,0) {$\bullet$};
\node at (4,0) {$\bullet$};
\node at (5,0) {$\bullet$};
\draw (5,0) arc [radius=2.5, start angle=0, end angle=180];
\draw (4,0) arc [radius=1.5, start angle=0, end angle=180];
\draw (3,0) arc [radius=0.5, start angle=0, end angle=180];

\node at (8,0) {$\bullet$};
\node at (9,0) {$\bullet$};
\node at (10,0) {$\bullet$};
\node at (11,0) {$\bullet$};
\node at (12,0) {$\bullet$};
\node at (13,0) {$\bullet$};
\draw (13,0) arc [radius=2.5, start angle=0, end angle=180];
\draw (12,0) arc [radius=0.5, start angle=0, end angle=180];
\draw (10,0) arc [radius=0.5, start angle=0, end angle=180];

\node at (16,0) {$\bullet$};
\node at (17,0) {$\bullet$};
\node at (18,0) {$\bullet$};
\node at (19,0) {$\bullet$};
\node at (20,0) {$\bullet$};
\node at (21,0) {$\bullet$};
\draw (19,0) arc [radius=1.5, start angle=0, end angle=180];
\draw (18,0) arc [radius=0.5, start angle=0, end angle=180];
\draw (21,0) arc [radius=0.5, start angle=0, end angle=180];

\node at (24,0) {$\bullet$};
\node at (25,0) {$\bullet$};
\node at (26,0) {$\bullet$};
\node at (27,0) {$\bullet$};
\node at (28,0) {$\bullet$};
\node at (29,0) {$\bullet$};
\draw (25,0) arc [radius=0.5, start angle=0, end angle=180];
\draw (29,0) arc [radius=1.5, start angle=0, end angle=180];
\draw (28,0) arc [radius=0.5, start angle=0, end angle=180];

\node at (32,0) {$\bullet$};
\node at (33,0) {$\bullet$};
\node at (34,0) {$\bullet$};
\node at (35,0) {$\bullet$};
\node at (36,0) {$\bullet$};
\node at (37,0) {$\bullet$};
\draw (33,0) arc [radius=0.5, start angle=0, end angle=180];
\draw (35,0) arc [radius=0.5, start angle=0, end angle=180];
\draw (37,0) arc [radius=0.5, start angle=0, end angle=180];
\end{tikzpicture}
\]

\begin{defn} \label{def-Catalan-pair}
Let $n$ be a positive integer. A \emph{Catalan-pair graph} on $n$ vertices is a graph $G$ that can be obtained by the following procedure. Start with $2n$ collinear points, of which we color $2k$ points red for some $0 \leq k \leq n$ and color the remaining points blue. Then, choose Catalan-arc matchings of sizes $k$ and $n-k$ and place them on the red and blue points, respectively, with the latter being faced downwards rather than upwards. Finally, construct a graph $G$ with one vertex for each of the $n$ arcs, where two vertices are adjacent if and only if the endpoints of the corresponding arcs alternate.
\end{defn}

As an example, we have the following Catalan-pair graph on $9$ vertices, where we colored the arcs according to the color of the points they connect. We say that the pair of Catalan-arc matchings on the left is a \emph{representative} for the graph on the right, or alternatively that it \emph{represents} the graph on the right.
\[
\begin{tikzpicture}[scale=0.5]
\draw (0,0) -- (17,0);
\node at (0,0) {\red $\bullet$};
\node at (1,0) {\red $\bullet$};
\node at (2,0) {\blue $\bullet$};
\node at (3,0) {\red $\bullet$};
\node at (4,0) {\blue $\bullet$};
\node at (5,0) {\blue $\bullet$};
\node at (6,0) {\red $\bullet$};
\node at (7,0) {\red $\bullet$};
\node at (8,0) {\red $\bullet$};
\node at (9,0) {\red $\bullet$};
\node at (10,0) {\red $\bullet$};
\node at (11,0) {\blue $\bullet$};
\node at (12,0) {\blue $\bullet$};
\node at (13,0) {\red $\bullet$};
\node at (14,0) {\blue $\bullet$};
\node at (15,0) {\red $\bullet$};
\node at (16,0) {\blue $\bullet$};
\node at (17,0) {\blue $\bullet$};
\draw[color=red] (3,0) arc [radius=1, start angle=0, end angle=180];
\node at (2,1.5) {$u_2$};
\draw[color=red] (6,0) arc [radius=3, start angle=0, end angle=180];
\node at (3,3.5) {$u_1$};
\draw[color=red] (8,0) arc [radius=0.5, start angle=0, end angle=180];
\node at (7.5,1) {$u_3$};
\draw[color=red] (13,0) arc [radius=1.5, start angle=0, end angle=180];
\node at (11.5,2) {$u_5$};
\draw[color=red] (15,0) arc [radius=3, start angle=0, end angle=180];
\node at (12,3.5) {$u_4$};
\draw[color=blue] (2,0) arc [radius=1, start angle=180, end angle=360];
\node at (3,-1.5) {$v_1$};
\draw[color=blue] (5,0) arc [radius=3, start angle=180, end angle=360];
\node at (8,-3.5) {$v_2$};
\draw[color=blue] (12,0) arc [radius=2.5, start angle=180, end angle=360];
\node at (14.5,-3) {$v_3$};
\draw[color=blue] (14,0) arc [radius=1, start angle=180, end angle=360];
\node at (15,-1.5) {$v_4$};

\node at (19,1) {$\bullet$};
\node at (19,1.5) {$u_1$};
\node at (21,1) {$\bullet$};
\node at (21,1.5) {$u_2$};
\node at (23,1) {$\bullet$};
\node at (23,1.5) {$u_3$};
\node at (25,1) {$\bullet$};
\node at (25,1.5) {$u_4$};
\node at (27,1) {$\bullet$};
\node at (27,1.5) {$u_5$};
\node at (20,-1) {$\bullet$};
\node at (20,-1.5) {$v_1$};
\node at (22,-1) {$\bullet$};
\node at (22,-1.5) {$v_2$};
\node at (24,-1) {$\bullet$};
\node at (24,-1.5) {$v_3$};
\node at (26,-1) {$\bullet$};
\node at (26,-1.5) {$v_4$};
\draw (19,1) -- (22,-1);
\draw (21,1) -- (20,-1);
\draw (25,1) -- (22,-1);
\draw (25,1) -- (24,-1);
\draw (25,1) -- (26,-1);
\draw (27,1) -- (22,-1);
\draw (27,1) -- (24,-1);
\end{tikzpicture}
\]

As a first observation, note that all of the arcs on the top are chosen to be non-intersecting, and similarly for all of the arcs on the bottom. Therefore, if the endpoints of two arcs alternate (and hence correspond to an edge in $G$) these arcs necessarily come from different sides. Thus every Catalan-pair graph is bipartite.

Catalan-pair graphs were recently introduced in \cite{Catalan-pair} where they were called paperclip graphs.  This paper primarily studied partial characterizations of Catalan-pair graphs, as well as bounds on the number of Catalan-pair graphs on a given number of vertices.  We note that Catalan-pair graphs can also equivalently be defined as bipartite circle graphs.  A circle graph is any graph whose vertices can be associated to a set of chords of a circle with two vertices being made adjacent if and only if their corresponding chords intersect. The equivalence between Catalan-pair graphs and bipartite circle graphs follows, similarly to the equivalence between Catalan-arc objects and sets of non-intersecting chords on a circle, by wrapping around the line containing the points and connecting it.

Circle graphs have been extensively studied, mainly from an algorithmic viewpoint.  For example, Spinrad \cite{Spinrad} produced an $O(n^2)$-time algorithm for identifying whether a given graph is a circle graph.  Many problems that are know to be NP-complete for general graphs turn out to have polynomial time algorithms when restricted to circle graphs.  Recently Tiskin showed that a maximum clique of a circle graph can be found in $O(n (\log n)^2)$ time \cite{Tiskin}, and Gregg and Nash have shown that a maximum independent set can be found in time $O(\al n)$, where $\al$ denotes the independence number of the circle graph \cite{Nash}.

The main purpose of this paper is to introduce a model to randomly generate a Catalan-pair graph on $n$ vertices, which we denote by $CP_n$, and to establish various properties about this random graph. Before we precisely define our random graph model, we briefly summarize our main results.

\begin{thm}\label{T-edges}
The expected number of edges of the random Catalan-pair graph $CP_n$ satisfies
\[
\E[e(CP_n)] \sim \frac1{\pi} n \log n.
\]
Moreover, for any $\epsilon > 0$ we asymptotically almost surely have $|e(CP_n) - \frac1{\pi} n \log n| < \epsilon n \log n$.
\end{thm}

We also obtain an asymptotic formula for the expected number of isolated vertices in $CP_n$.
\begin{thm}\label{T-isolated}	
Let $I_n$ denote the number of isolated vertices in $CP_n$. Then
\[
	\E[I_n] \sim \gamma n,
\]
where $\gamma$ is the constant defined by
\[
\gamma = 4 \sum_{m=1}^{\infty} 16^{-m} \sum_{b=0}^{m-1} \binom{2m-2}{2b} C_{m-1-b} C_b  = 0.3023\ldots.
\]
Moreover, for any $\epsilon > 0$ we asymptotically almost surely have $|I_n - \gamma n| < \epsilon n$.
\end{thm}

In addition to this, we deduce the order of magnitude for the expected number of (induced) subgraphs of any connected Catalan-pair graph with at least three vertices.  To this end, Let $N_H(G)$ denote the number of subgraphs of $G$ that are isomorphic to $H$ and let $N_H^*(G)$ denote the number of induced subgraphs of $G$ that are isomorphic to $H$.

\begin{thm}\label{T-induced}
Let $H$ be a connected Catalan-pair graph on $v \geq 3$ vertices. The expected number of (induced) subgraphs of the random Catalan-pair graph $CP_n$ isomorphic to $H$ satisfies
\[
\E[N_H(CP_n)] = \Theta(n^{v/2}).
\]
\[
\E[N^*_H(CP_n)] = \Theta(n^{v/2}).
\]
\end{thm}

The outline of the paper is as follows. In Section~\ref{sec-random} we will define our model to randomly generate Catalan-pair graphs. This model requires us to randomly select a Catalan-arc matching, and in Section~\ref{sec-catalan} we will derive some technical lemmas related to this step. In Section~\ref{sec-edges} we will determine the asymptotic behavior of the expected number of edges, and in Section~\ref{sec-isolated} we will determine the expected number of isolated vertices and show the desired concentration result by bounding the variance of the number of isolated vertices. In Section~\ref{sec-edgevariance} we will similarly bound the variance of the number of edges, with a large part of the proof deferred to Appendix~\ref{appendix-edgevariance}. Section~\ref{sec-induced} will focus on proving Theorem~\ref{T-induced}, and along the way we will prove a more general lower bound for unconnected Catalan-pair graphs. We will end that section with a general result on the connected components of $CP_n$. In Section~\ref{sec-computational} we will discuss experimental data obtained by randomly generating Catalan-pair graphs of various sizes, after which we will end with some final remarks and possible future problems in Section~\ref{sec-conclusion}.

We collect some notation and definitions that we will use throughout the text.  For $1 \leq a < b \leq 2n$, we say that $(a,b)$ \emph{match} if the $a^{\textrm{th}}$ and $b^{\textrm{th}}$ point have the same color and if there is an arc connecting these two points. In the earlier example, the matching pairs are $(1,7)$, $(2,4)$, $(3,5)$, $(6,12)$, $(8,9)$, $(10,16)$, $(11,14)$, $(13,18)$ and $(15,17)$. We similarly say that $(a,b)$ match in a single Catalan-arc matching of size $n$ if there is an arc connecting these two points. For $1 \leq a < b < c < d \leq 2n$ we say that $(a,b,c,d)$ is \emph{an edge} if $(a,c)$ and $(b,d)$ match. For example, in the the graph from before $(6,10,12,16)$ is an edge, and it corresponds to the edge between $u_4$ and $v_2$. We say that an arc in a single Catalan-pair matching has length $k$ if it covers $k-1$ smaller arcs, or equivalently if the two points it connects have $2k-2$ points between them.

\section{Random Catalan-pair graphs} \label{sec-random}
In this section we define a model to generate a random Catalan-pair graphs on $n$ vertices. Consider the following procedure, starting with $2n$ collinear points.
\begin{itemize}
\item[1.] For each of the first $2n-1$ points, uniformly and independently color each of these points either red or blue. Then color the last point red or blue, whichever makes it so that the total number of points of each color is even.
\item[2.] Suppose that we have $2k$ red points, and consequently $2(n-k)$ blue points. Independently and uniformly pick Catalan-arc matchings of size $k$ and $n-k$ from the set of all possible Catalan-arc matchings of that size, and place these above and below the red and blue points respectively.
\item[3.] Create a graph according to Definition~\ref{def-Catalan-pair}, and denote this (random) graph by $CP_n$.
\end{itemize}

One of the advantages of this model is that with high probability roughly half of the points (or any large enough subset of the points for that matter) will be colored red. This is an immediate consequence of the following concentration result, which can be found in a slightly different form in \cite[Cor. A.1.2.]{AlonSpencer}.
For $1\le i\le n$, let $X_i$ denote mutually independent random variables with $\Pr[X_i=1]=\Pr[X_i=0]=\frac{1}{2}$, and define $S_n=\sum_{i=1}^n X_i$.  For $a>0$,
\begin{equation}\label{eq-Conc}
		\Pr[|S_n-n/2|>a]<2e^{-2a^2/n}.
\end{equation}
Note that because of the forced choice of the color of the last point, our setting is not completely identical to that of the above result. However, it does apply for any proper subset of the points, and the concentration result for the total number of points of a given color is almost unaffected.

\section{Random Catalan matchings} \label{sec-catalan}
To generate $CP_n$ we must choose a random Catalan-arc matching from all such matchings of a given size. In this section we compute the probability of having a given set of arcs connecting a given set of points within this randomly chosen Catalan-arc matching.  We note that studying the structure of a random object enumerated by the Catalan numbers is of independent interest, and other work in this direction has been done in, for example, \cite{RC} and \cite{AC}.

Let $\mathcal{C}_n$ denote the set of Catalan-arc matchings of size $n$, and let $C_n = |\mathcal{C}_n| = \frac1{n+1} \binom{2n}{n}$ be the $n$th Catalan number. We recall the asymptotic formula
\begin{equation}\label{eq-CatAsy}
C_n\sim \f{4^n}{\sqrt{\pi}n^{3/2}},
\end{equation}
which can be derived, for example, by Stirling's formula.

Throughout this section, let $C$ be a Catalan-arc matching chosen uniformly from $\mathcal{C}_n$. As mentioned, we are interested in the probability of having a given set of arcs connecting a given set of points within $C$. It is clear that in order for this to be able to happen, the points and arcs have to satisfy some conditions. First of all the endpoints of any given arc must have an even number of points between them, since any arc connecting at least one of these points must connect two of these points. Additionally, it is clear that none of the given arcs are allowed to intersect.

This leads to the following definition, where one should think of having specified arcs connecting points $x_i$ and $x_i+2k_i-1$ for all $i$.

\begin{defn}
Let $\b{x} = (x_1,\ldots,x_s)$ and $\b{k}=(k_1,\ldots,k_s)$ be $s$-tuples of positive integers with $x_1 < \ldots < x_s$. We say that $(\b{x},\b{k})$ is a \emph{valid pair} if
\begin{itemize}
\item[1.] For all $i$ we have $1 \leq x_i < x_i + 2k_i - 1 \leq 2n$.
\item[2.] The integers $x_1, x_1 + 2k_1 - 1, \ldots, x_s, x_s + 2k_s - 1$ are all distinct.
\item[3.] There are no $i \neq j$ with $x_i < x_j < x_i + 2k_i - 1 < x_j + 2k_j - 1$.
\end{itemize}
\end{defn}
As an example, for $n=8$, we have the valid pair $((2,4),(5,2))$ which we think of as having specified arcs connecting points $2$ and $11$ and $4$ and $7$.
\[
\begin{tikzpicture}[scale=0.4]
\node at (1,0) {$\bullet$};
\node at (2,0) {$\bullet$};
\node at (3,0) {$\bullet$};
\node at (4,0) {$\bullet$};
\node at (5,0) {$\bullet$};
\node at (6,0) {$\bullet$};
\node at (7,0) {$\bullet$};
\node at (8,0) {$\bullet$};
\node at (9,0) {$\bullet$};
\node at (10,0) {$\bullet$};
\node at (11,0) {$\bullet$};
\node at (12,0) {$\bullet$};
\node at (13,0) {$\bullet$};
\node at (14,0) {$\bullet$};
\node at (15,0) {$\bullet$};
\node at (16,0) {$\bullet$};
\draw (11,0) arc [radius=4.5, start angle = 0, end angle = 180];
\draw (7,0) arc [radius=1.5, start angle = 0, end angle = 180];
\end{tikzpicture}
\]
As mentioned before, the conditions imposed on $(\b{x},\b{k})$ are necessary for there to be a Catalan-arc matching with arcs on these specified positions. In this case, it is not so hard to see that we can indeed extend this to a Catalan-arc matching, for example as follows.
\[
\begin{tikzpicture}[scale=0.4]
\node at (1,0) {$\bullet$};
\node at (2,0) {$\bullet$};
\node at (3,0) {$\bullet$};
\node at (4,0) {$\bullet$};
\node at (5,0) {$\bullet$};
\node at (6,0) {$\bullet$};
\node at (7,0) {$\bullet$};
\node at (8,0) {$\bullet$};
\node at (9,0) {$\bullet$};
\node at (10,0) {$\bullet$};
\node at (11,0) {$\bullet$};
\node at (12,0) {$\bullet$};
\node at (13,0) {$\bullet$};
\node at (14,0) {$\bullet$};
\node at (15,0) {$\bullet$};
\node at (16,0) {$\bullet$};
\draw[line width=0.4mm] (11,0) arc [radius=4.5, start angle = 0, end angle = 180];
\draw[line width=0.4mm] (7,0) arc [radius=1.5, start angle = 0, end angle = 180];
\draw[color=red] (16,0) arc [radius=7.5, start angle = 0, end angle = 180];
\draw[color=red] (15,0) arc [radius=0.5, start angle = 0, end angle = 180];
\draw[color=red] (13,0) arc [radius=0.5, start angle = 0, end angle = 180];
\draw[color=red] (10,0) arc [radius=3.5, start angle = 0, end angle = 180];
\draw[color=red] (9,0) arc [radius=0.5, start angle = 0, end angle = 180];
\draw[color=red] (6,0) arc [radius=0.5, start angle = 0, end angle = 180];
\end{tikzpicture}
\]
Below we will see that the condition of $(\bf{x},\bf{k})$ being a valid pair is also a sufficient condition to have a Catalan-arc matching with arcs connecting $x_i$ and $x_i+2k_i-1$. In fact, we will determine the explicit probability of having arcs on these given positions. To this end, let $A(\b{x},\b{k})$ denote the event that $(x_i,x_i+2k_i-1)$ match in $C$ for all $i$.

Before we can determine the probability of this happening we need some notation. In the above example, we see that in order to extend to a Catalan-arc matching, we have to connect the two points within the smaller arc, we have to connect the four points within the larger arc (but outside of the smaller arc), and finally we have to connect the six points outside of the larger arc. Below we define integers that are analogues of the two, four, and six above.

For a valid pair $(\b{x},\b{k})$ and $1 \leq i \leq s$, let $M_i$ be the set of $x$ such that $x_i<x<x_i+2k_i-1$ and such that there exists no $j\ne i$ with $x_j\le x\le x_j+2k_j-1$. We let $M_0$ be the set of $x$ such that $1\le x\le 2n$ and such that there exists no $i$ with $x_i\le x\le x_i+2k_i-1$.  Observe that every $x$ with $1\le x\le 2n$ is either of the form $x_i$ or $x_i + 2k_i-1$ for some $i$, or else belongs to a unique $M_i$. Furthermore, it is easy to see that each $M_i$ has an even (possibly $0$) number of elements, so the numbers $m_i = |M_i|/2$ are nonnegative integers, and from the definition it follows that these numbers sum to $n - s$.

We can now explicitly compute the probability that $(x_i,x_i+2k_i-1)$ match in $C$ for all $i$.
\begin{lem}\label{L-CatProb}
If $(\b{x},\b{k})$ is a valid pair, then
\begin{equation*}\label{E-Exp}
\Pr[A(\b{x},\b{k})]=\frac{1}{C_n} \cdot \prod_{i=0}^s C_{m_i}.
\end{equation*}
\end{lem}

\begin{proof}
Since each Catalan-arc matching is chosen with probability $\frac1{C_n}$, it suffices to show that there are $\prod_{i=0}^s C_{m_i}$ Catalan-arc matchings for which $(x_i,x_i+2k_i-1)$ match for all $i$. We show that a Catalan-arc matching satisfies this condition if and only if points in some $M_i$ are only connected to points in that $M_i$ and the set of arcs on the points in $M_i$ is a Catalan-arc matching.

First, assume that there exists $i \neq j$ such that there is a Catalan-arc matching that connects a point $x$ in $M_i$ to a point $y$ in $M_j$. Without loss of generality we may assume that $j \neq 0$ and that we do not have $x_j < x_i < x_i+2k_i-1 < x_j+2k_j-1$ (if the latter happens, simply switch $i$ and $j$). This implies that $x_j < y < x_j + 2k_j-1$ and $x \not \in [x_j,x_j+2k_j-1]$, but then the arc connecting $x$ and $y$ would intersect the arc connecting $x_j$ and $x_j+2k_j-1$, a contradiction. Furthermore, it is clear that the induced set of arcs on the points in $M_i$ still has no intersecting arcs.

Conversely, suppose we choose Catalan-arc matchings to go on the points of each $M_i$. By definition, there do not exist points $a < b < c < d$ with $a,c \in M_i$ and $b, d \in M_j$ for $i \neq j$, so arcs in $M_i$ and $M_j$ will not intersect when $i \neq j$, and clearly also not for $i = j$. Lastly, points in $M_i$ either lie completely inside an interval $[x_j,x_j+2k_j-1]$ or lie completely outside of it, so arcs on the $M_i$ will also not intersect arcs of the form $(x_j,x_j+2k_j-1)$.

Therefore, since a Catalan-arc matching on the points of $M_i$ has $m_i$ arcs, there are $C_{m_i}$ choices for this matching. Since these choices can be made independently, the total number of desired Catalan-arc matchings equals $\prod_{i=0}^s C_{m_i}$, as desired.
\end{proof}

By combining \eqref{eq-CatAsy} and Lemma~\ref{L-CatProb} we can obtain bounds for this probability.
\begin{cor}\label{C-LargeCat}
Let $(\b{x},\b{k})$ be a valid pair.  There exist positive real numbers $\alpha_s$, $\beta_s$ such that
\[
\alpha_s \f{n^{3/2}}{\prod' m_i^{3/2}} \leq \Pr[A(\b{x},\b{k})] \leq \beta_s \f{n^{3/2}}{\prod' m_i^{3/2}},
\]
where $\prod'$ indicates the product over all $0 \leq i \leq s$ with $m_i \neq 0$.
\end{cor}
\begin{proof}
Let us prove the lower bound, the proof for the upper bound is analogous. Because of the asymptotic formula in \eqref{eq-CatAsy} there exist positive numbers $a < 1 < A$ such that
\begin{equation}\label{eq-aA}
a \f{4^n}{\sqrt{\pi}n^{3/2}} \leq C_n \leq A \f{4^n}{\sqrt{\pi}n^{3/2}}
\end{equation}
for all $n \geq 1$. Since $C_0 = 1$ we find
\begin{align*}
\Pr[A(\b{x},\b{k})] &= \frac{1}{C_n} \cdot \prod_{i=0}^s C_{m_i} = \frac{1}{C_n} \cdot \sideset{}{'}\prod C_{m_i} \\
    &\geq \frac{\sqrt{\pi} n^{3/2}}{A \cdot 4^n} \cdot \sideset{}{'}\prod \frac{a \cdot 4^{m_i}}{\sqrt{\pi} m_i^{3/2}} \geq 4^{\sum' m_i-n} \cdot \frac{ a^{s+1}}{A\cdot \pi^{s/2}} \frac{n^{3/2}}{\prod' m_i^{3/2}} = \alpha_s \f{n^{3/2}}{\prod' m_i^{3/2}},
\end{align*}
where we use that $\sum' m_i = \sum_{i=0}^s m_i = n-s$.
\end{proof}

\section{The expected number of edges} \label{sec-edges}

In this section we will determine the asymptotic behavior of the expected number of edges of $CP_n$. To this end, we start by establishing a general upper bound on the probability that $CP_n$ contains a given structure on a given set of points.

We consider two analogues of the valid pairs introduced in Section~\ref{sec-catalan}. Let $\b{x}=(x_1,\ldots,x_s),\ \b{k}=(k_1,\ldots,k_s),\ \b{y}=(y_1,\ldots,y_t),\ \b{l}=(\ell_1,\ldots,\ell_t)$ be tuples of positive integers with $x_1 < \ldots < x_s$ and $y_1 < \ldots < y_t$. We say that this quadruple is \emph{valid} if for all $1 \leq i \leq s$ and $1 \leq j \leq t$ we have $1 \leq x_i < x_i + k_i \leq 2n$ and $1 \leq y_j < y_j + \ell_j \leq 2n$, and if there exists at least one representative for a Catalan-pair graph on $n$ vertices for which $(x_i,x_i+k_i)$ and $(y_j,y_j+\ell_j)$ match for all $i,j$.

Similarly, we say that such a quadruple $(\b{x},\b{k},\b{y},\b{l})$ is \emph{good} if
\begin{itemize}
\item[1.] $1 \leq x_i < x_i + k_i \leq 2n$ and $1 \leq y_j < y_j + \ell_j \leq 2n$ for all $1 \leq i \leq s$ and $1 \leq j \leq t$.
\item[2.] Any two numbers of the form $x_i, x_i + k_i, y_j$ or $y_j + k_j$ differ by at least $2$.
\item[3.] There exists no $i \ne j$ such that $x_i<x_j< x_i+k_i< x_j+k_j$ or $y_i<y_j< y_i+\ell_i< y_j+\ell_j$.
\end{itemize}
In the proof of Lemma~\ref{L-GenLowProb} we will see that these conditions imply that there exists a representative for a Catalan-pair graph $G$ such that $(x_i,x_i+k_i)$ and $(y_i,y_i+\ell_i)$ match for all $1 \leq i \leq s$ and all $1 \leq j \leq t$. Therefore, any good quadruple is also a valid quadruple.

Given a valid quadruple $(\b{x},\b{k},\b{y},\b{l})$, we would like to have an analogue of the integers $m_i$ defined in Section~\ref{sec-catalan}. To this end, for $1 \leq i \leq s$,  set $f_i$ to be the number of $x_i < x < x_i+k_i$ such that there is no $i'$ with $x_{i'} \leq x \leq x_{i'}+k_{i'}$ and such that $x$ is not of the form $y_j$ or $y_j + \ell_j$ for any $j$. Set $f_0$ to be the number of $1 \leq x \leq 2n$ that do not belong to any interval $[x_i,x_i+k_i]$, nor are of the form $y_j$ or $y_j + \ell_j$. Similarly define $g_0, g_1, \ldots, g_t$.

Let $(\b{x},\b{k},\b{y},\b{l})$ be a valid quadruple where $\b{x}$ and $\b{y}$ have length $s$ and $t$ respectively. Let $A(\b{x},\b{k},\b{y},\b{l})$ denote the intersection of the following events.
\begin{itemize}
\item[1.] The points $x_i$ and $x_i+k_i$ are colored red and the points $y_j$ and $y_j+\ell_j$ are colored blue for all $i,j$.
\item[2.] For all $i$ and $j$ the number of red points $x$ with $x_i<x<x_i+k_i$ and the number of blue points $y$ with $y_j<y<y_j+\ell_j$ is even.
\item[3.] For all $i$ and $j$ we have that $(x_i,x_i+k_i)$ and $(y_j,y_j+\ell_j)$ match in $CP_n$.
\end{itemize}

We would like to point out that the second condition is necessary for $(x_i,x_i+k_i)$ and $(y_j,y_j+\ell_j)$ to match for all $i$ and $j$. Therefore, we could technically omit this condition, but we have included it to improve the readability of our proofs.

We have the following upper bound for the probability that $A(\bf{x},\bf{k},\bf{y},\bf{l})$ occurs.
\begin{lem}\label{L-upperbound}
There exists a positive real number $\beta_{s,t}$ such that for any valid quadruple $(\b{x},\b{k},\b{y},\b{l})$ (with $\b{x}$ and $\b{y}$ of length $s$ and $t$ respectively) and sufficiently large $n$ we have
\[
\Pr[A(\b{x},\b{k},\b{y},\b{l})] \leq \beta_{s,t} n^3 \cdot \widetilde{\prod_i} f_i^{-3/2} \cdot \widetilde{\prod_j} g_j^{-3/2},
\]
where $\widetilde{\prod}$ indicates the product over all $i$ and $j$ for which $f_i, g_j \geq 16 (s+t) \log n$.
\end{lem}

\begin{proof}
Let $v = (s+t)$. Note that with probability $2^{-2(s+t)} = 4^{-v}$ all of $x_i, x_i+k_i, y_j, y_j+\ell_j$ have the correct color. From now on we condition on this event happening. For each $0 \leq i \leq s$, let $2r_i$ denote the number of points counted by $f_i$ which are colored red, where we note that $r_i$ may not be an integer. For each $i$ with $f_i \geq 16 v \log n$, we use \eqref{eq-Conc} to conclude that
\[
\Pr[|2r_i-f_i/2|>\sqrt{v f_i \log n}] < 2 n^{-2v}.
\]
Note that if $|2r_i-f_i/2|\leq\sqrt{v f_i \log n}$, then in particular we have $2r_i \geq f_i/2 - \sqrt{v f_i \log n} \geq f_i/4$, where we used $f_i \geq 16 v \log n$ in the last step. Therefore, with probability at most $2(v+2)n^{-2v}$ we have $r_i < f_i/8$ or $b_j < g_j/8$ for some $i$ or $j$ for which $f_i, g_j \geq 16 v \log n$.

Let $B_n$ and $R_n$ be the total number of blue and red points respectively. We condition on the event that $r_i \geq f_i/4$ and $b_j \geq g_j/4$ for all $i$ and $j$ for which $f_i, g_j \geq 16 v \log n$. If any of the numbers $r_i, b_j$ is not an integer, or equivalently if the number of red/blue points in some appropriate region is not even, the probability that all of $(x_i,x_i+k_i)$ and $(y_j,y_j+\ell_j)$ match is $0$, which is fine since we are only interested in an upper bound on the probability. If all $r_i$, $b_j$ are integers we can apply Corollary~\ref{C-LargeCat} to show that the probability that all of the $(x_i,x_i+k_i)$ and $(y_j,y_j+\ell_j)$ match is at most
\begin{align*}
\beta_s R_n^{3/2} \cdot \widetilde{\prod_i} r_i^{-3/2} \cdot \beta_t B_n^{3/2} \cdot \widetilde{\prod_j} b_j^{-3/2} &\leq \beta_s \cdot (2n)^{3/2} \cdot \widetilde{\prod_i} (f_i/8)^{-3/2} \cdot \beta_t (2n)^{3/2} \cdot \widetilde{\prod_j} (g_j/8)^{-3/2} \\
    &= O\left(n^3 \cdot \widetilde{\prod_i} f_i^{-3/2} \cdot \widetilde{\prod_j} g_j^{-3/2}\right),
\end{align*}
where in the first expression we ignored all $i$ for which $f_i < 16 v \log n$ since in the formula of Corollary~\ref{C-LargeCat} these terms either do not appear, or they contribute a multiplicative factor of the form $x^{-3/2}$ for some $x \geq 1$, hence leaving it out will still yield an upper bound.

Therefore, we know that
\[
\Pr[A(\b{x},\b{k},\b{y},\b{l})] \leq 4^{-v} \cdot \left(2(v+2)n^{-2v} + O\left(n^3 \cdot \widetilde{\prod_i} f_i^{-3/2} \cdot \widetilde{\prod_j} g_j^{-3/2}\right)\right).
\]
Since $f_i \leq 2n$, $g_i \leq 2n$, and since the above products contain at most $s+1$ and $t+1$ terms respectively, we find
\[
n^3 \cdot \widetilde{\prod_i} f_i^{-3/2} \cdot \widetilde{\prod_j} g_j^{-3/2} \geq n^3 (2n)^{-3/2(s+1+t+1)} = 2^{-3/2(v+2)} \cdot n^{3 - 3/2(v+2)} \gg n^{-2v}
\]
for sufficiently large $n$, and hence the $n^3 \cdot \widetilde{\prod_i} f_i^{-3/2} \cdot \widetilde{\prod_j} g_j^{-3/2}$ term dominates this expression.
\end{proof}

Using similar ideas, we can deduce an upper bound on the expected number of arcs in $CP_n$ whose lengths lie in a specific range.
\begin{lem}\label{L-Range}
		For any $1\le \al \le \be\le 2n$, let $A_{\al,\be}$ denote the number of matching arcs in $CP_n$ of the form $(i,i+k)$ with $\al\le k\le \be$.  Then \[\E[A_{\al,\be}]=O(\al^{-1/2}n+\be ne^{-\al/16}).\]
		
		In particular, if $32\log n\le \al$ we have
		\[
		\E[A_{\al,\be}]=O(\al^{-1/2}n).
		\]
	\end{lem}
	\begin{proof}
		We first consider some reductions of the problem.  If $\al=O(1)$ the bound is trivial, so we will assume that $\al=\om(1)$. For any $\alpha\ge n$ the proposed bound is $O(\sqrt{n})$, so we can assume without loss of generality that $\alpha\le n$.  Also, for any $k\ge 2n-32\log n$, $CP_n$ contains at most two non-intersecting arcs of length $k$ (one for each color) since $k>n$.  Thus we can assume that $\be\le 2n-32\log n$, which will cause $\E[A_{\al,\be}]$ to decrease by at most $2\cdot 32\log n=O(\al^{-1/2}n)$ when $\al\le n$.
		
		For $\al\le k\le \be$, let $A(i,k)$ denote the event that $(i,i+k)$ matches in $CP_n$. Let $2r_1$ denote the number of points $x$ in $i<x<i+k$ colored red and let $2r_2$ denote the number of points $x$ with $x<i$ or $x>i+k$ colored red, where as before we note that $r_1$ or $r_2$ may not be an integer.  The probability that either $|2r_1-(k-1)/2|>(k-1)/4$ or $|2r_2-(2n-k-1)/2|>(2n-k-1)/4$ is at most $e^{-(k-1)/8}+e^{-(2n-k-1)/8}$.  Conditional on neither of these events occurring, we can proceed as in Lemma~\ref{L-upperbound} and find that the probability of $(i,i+k)$ matching is at most $c n^{3/2}(k-1)^{-3/2}(2n-(k+1))^{-3/2}$ for some absolute constant $c$.  In total then we have that \[\Pr[A(i,k)]\le c n^{3/2}(k-1)^{-3/2}(2n-(k+1))^{-3/2}+e^{-(k-1)/8}+e^{-(2n-k-1)/8}.\]  Moreover, we have that $\Pr[A(i,k)]=0$ for $i> 2n-k$.  Because
\[       E[A_{\al,\be}]=\sum_{k=\al}^\be \sum_{i=1}^{2n} \Pr[A(i,k)],
\]
we have that
\begin{equation}\label{EQ-Arcs}\E[A_{\al,\be}]\le\sum_{k=\al}^\be (2n-k)(c n^{3/2}(k-1)^{-3/2}(2n-(k+1))^{-3/2}+e^{-(k-1)/8}+e^{-(2n-k-1)/8}).
		\end{equation}
		
		Let $\gam=\min(\be,n)$.  For $\al\le k\le \gam$ and $n$ sufficiently large, we have that $(2n-(k+1))\ge \half n$ and $(k-1)\ge\half k$.  Thus the terms in \eqref{EQ-Arcs} are at most \[2n(2^{-3}ck^{-3/2}+2e^{-\al/16}+2e^{-n/16})\le 2^{-2}cnk^{-3/2}+4ne^{-\al/16}.\]  Thus \eqref{EQ-Arcs} restricted to this range is at most
		\[
			\sum_{k=\al}^\gam 2^{-2}cnk^{-3/2}+4ne^{-\al/16}\le 2^{-2}n \int_{\al-1}^\infty cx^{-3/2}dx+4\gam ne^{-\al/16}=O(\al^{-1/2}n+\be ne^{-\al/16}).
		\]
		
		If $\be\le n$ then this completes the proof.  Otherwise we can assume $\be=2n-32\log n$.  Using similar logic as before, for $n\le k\le 2n-32\log n$ we have that the terms of \eqref{EQ-Arcs} are at most \[2^{-2}c(2n-k)^{-1/2}+4ne^{-2\log n}=2^{-2}c(2n-k)^{-1/2}+4n^{-1}.\]
		
		Again summing over the relevant range and bounding our sum with an integral gives an upper bound for \eqref{EQ-Arcs} in this range of
		\[
			\sum_{k=n}^{2n-32\log n} (2n-k)^{-1/2}+4n^{-1}=O(\sqrt{n})=O(\al^{-1/2}n).
		\]
		
		Summing the contributions from these ranges gives the desired result. \end{proof}

\subsection{The expected number of edges} \label{subsec-edges}

We are now ready to prove the first part of Theorem~\ref{T-edges}.  We will do so by showing that for any $\epsilon > 0$ we have
\begin{equation}\label{eq-edges}
(1-\epsilon)\frac1{\pi} n \log n +  o(n \log n) \leq \E[e(CP_n)] \leq (1+\epsilon) \frac1{\pi} n \log n + o(n \log n).
\end{equation}

It is clear that
\begin{equation}\label{eq-exp-edges}
\E[e(CP_n)] = \sum \Pr[A(x,k,y,\ell)],
\end{equation}
where the sum is over all valid quadruples $(x,k,y,\ell)$ of positive integers such that $1 \leq x < y < x+k < y+\ell \leq 2n$ or $1 \leq y < x < y+\ell < x+k \leq 2n$.

We break up this sum into various parts, and we will show that all but one will contribute $o(n \log n)$, and that the remaining part will contribute between $(1-\epsilon) \frac1{\pi} n \log n$ and $(1+\epsilon) \frac1{\pi} n \log n$. Let $c < 1$ be a positive real number and $d$ be a positive integer, where eventually we will pick $c$ small and $d$ large to get our bounds within the desired $(1 \pm \epsilon)$ region.

\begin{prop} \label{prop-edges}
Consider the contribution to \eqref{eq-exp-edges} coming from each of the following subsets of the quadruples.
\begin{itemize}
\item[(i)] Valid quadruples $(x,k,y,\ell)$ with $k < d \log n$ or $\ell < d \log n$.
\item[(ii)] Valid quadruples $(x,k,y,\ell)$ with $k > 2n - d \log n$ or $\ell > 2n - d \log n$.\item[(iii)] Quadruples $(x,k,y,\ell)$ with $d \log n \leq k,\ell \leq 2n - d \log n$ that are valid but not good.
\item[(iv)] Good quadruples $(x,k,y,\ell)$ with $d \log n \leq k \leq cn < \ell \leq 2n - d \log n$ or $d \log n \leq \ell \leq cn < k \leq 2n - d \log n$.
\item[(v)] Good quadruples $(x,k,y,\ell)$ with $cn < k,\ell \leq 2n - d \log n$.
\end{itemize}
Each of these contributions is $o(n \log n)$.
\end{prop}

\begin{proof}
\begin{itemize}
\item[(i)] This contribution counts the expected number of edges that come from pairs of arcs with at least one arc of length at most $d \log n$. We first show that the number of such edges with at least one arc of length at most $\sqrt{\log n}$ is of order $o(n \log n)$ in any Catalan-pair graph, and therefore also in expectation. Indeed, any arc of length at most $\sqrt{\log n}$ has degree at most $\sqrt{\log n}$ since every interlacing arc must have one of its endpoints within the given arc. Since we have at most $n$ arcs of length at most $\sqrt{\log n}$, the total number of such edges is at most $n \sqrt{\log n} = o(n \log n)$.

    Now consider the edges involving an arc of length between $\sqrt{\log n} $ and $d \log n$. By Lemma~\ref{L-Range} there are at most $O(n(\log n)^{-1/4}+ \log n \cdot n e^{-\sqrt{\log n}/16})=o(n)$ such arcs in expectation.  Since each such arc can be involved in at most $d\log n$ edges, we conclude that the total expected number of edges involving vertices of this type is at most $o(n\log n)$.

\item[(ii)] This contribution counts the expected number of edges that come from a pair of arcs where at least one of the arcs has length larger than $2n - d \log n$. We show that the number of such arcs is $O((\log n)^2) = o(n \log n)$ for any Catalan-pair graph, which implies the same bound for the expected number of such edges. First, note that for $n$ large enough and each $N > 2n - d \log n$ there is at most one arc of length $N$ on either side.  Indeed, since $2n - d \log n > n$ for $n$ large enough, if we had two arcs of length $N$ on one side this would contradict the condition that the arcs do not intersect. Therefore, there are at most $2d \log n$ arcs of length at least $2n - d \log n$. Furthermore, each such arc interlaces with at most $d \log n$ arcs on the opposite side. Indeed, any such interlacing arc must have one of its endpoints outside the arc in question, and there are at most $d \log n$ such points. Therefore, we have at most $2 d \log n \cdot d \log n = O((\log n)^2)$ such edges, as desired.
\item[(iii)] We assume $d > 32$ in order to apply Lemma~\ref{L-upperbound}.

We know that for any $(x,k,y,\ell)$ in this range we have $\Pr(A(x,k,y,\ell)) = O(n^3 (2n-(k+2))^{-3/2} k^{-3/2} (2n-(\ell+2)^{-3/2} \ell^{-3/2})$. Furthermore, given $k$ and $\ell$ we claim that there are at most $16n$ quadruples $(x,k,y,\ell)$ that are valid but not good. This follows since there are at most $2n$ possibilities for $x$, and given $x$ we must have that $y$ or $y+\ell$ belongs to $\{x \pm 1, x+k\pm 1\}$.

Therefore, the total contribution is at most of the order of
\[
n^4 \sum_{k,\ell} (2n-(k+2))^{-3/2} k^{-3/2} (2n-(\ell+2))^{-3/2} \ell^{-3/2} = n^4 \left(\sum_k (2n-(k+2))^{-3/2} k^{-3/2}\right)^2.
\]
We can break up $\sum_k (2n-(k+2))^{-3/2} k^{-3/2}$ in the regions $k \leq n$ and $k > n$. When $k \leq n$ we have $(2n-(k+2))^{-3/2} \leq (n-2)^{-3/2}$, hence the contribution is at most $(n-2)^{-3/2} \sum_{k} k^{-3/2} =  O(n^{-3/2})$, since the sum of $k^{-3/2}$ is bounded. By a similar reasoning the other contribution is $O(n^{-3/2})$, so
\[
n^4 \left(\sum_k (2n-(k+2))^{-3/2} k^{-3/2}\right)^2 = n^4 O(n^{-3/2})^2 = O(n) = o(n \log n),
\]
as was to be shown.
\item[(iv)] Again we assume $d > 32$. Also, we only consider the case $d \log n \leq k \leq cn < \ell \leq 2n - d \log n$, the other case is analogous.

We claim that for given $k$ and $\ell$ there are at most $(2n-\ell) \cdot 2k$ good quadruples $(x,k,y,\ell)$. This holds since $y$ has to satisfy $y+\ell \leq 2n$, and after choosing $y$ we must have that $y-k \leq x \leq y-1$ or $y+\ell-k \leq x \leq y+\ell-1$, leaving at most $2k$ choices for $x$. Therefore, this region contributes at most
    \[
    \sum_{k=d \log n}^{cn} \sum_{\ell=cn}^{2n-d\log n} (2n-\ell) \cdot 2k \cdot n^3 (2n-(k+2))^{-3/2} k^{-3/2} (2n-(\ell+2))^{-3/2} \ell^{-3/2}.
    \]
    Note that this sum breaks up as
    \[
    2 n^3 \left(\sum_{k=d \log n }^{cn} k^{-1/2} (2n-(k+2))^{-3/2}\right) \cdot \left(\sum_{\ell=cn}^{2n-d \log n} \ell^{-3/2} (2n-\ell) \cdot (2n-(\ell+2))^{-3/2}\right).
    \]
    Using $(2n-(k+2))^{-3/2} \leq 2^{3/2} n^{-3/2}$ we find that
    \begin{align*}
    \sum_{k=d \log n}^{cn} k^{-1/2} (2n-(k+2))^{-3/2} &= O(n^{-3/2}) \cdot \sum_{d \log n+2}^{cn} k^{-1/2} \\
        & = O(n^{-3/2}) \cdot O(n^{1/2}) = O(n^{-1})
    \end{align*}
    where the second equality follows from comparison of the sum with an integral. An analogous computation shows that
    \[
    \sum_{\ell=cn}^{2n-d \log n} \ell^{-3/2} (2n-\ell) \cdot (2n-(\ell+2))^{-3/2} = O(n^{-1}),
    \]
    and therefore this range of $k$ and $\ell$ contributes at most $2n^3 \cdot O(n^{-1}) \cdot O(n^{-1}) = O(n)$, which is in particular $o(n \log n)$ as desired.
\item[(v)] Again we estimate the number of good quadruples $(x,k,y,\ell)$ for given $k$, $\ell$. Similar to above we have at most $(2n-k)$ and $(2n-\ell)$ choices for $x$ and $y$ respectively, and therefore we have at most $(2n-k)(2n-\ell)$ good quadruples in total. Thus this part of the sum contributes at most
    \[
    \sum_{k,\ell=cn}^{2n-d\log n} (2n-k) (2n-\ell) \cdot n^3 (2n-(k+2))^{-3/2} k^{-3/2} (2n-(\ell+2))^{-3/2} \ell^{-3/2}.
    \]
As in case 4, this factors as
    \[
    n^3 \left(\sum_{k=cn}^{2n-d \log n} k^{-3/2} (2n-k) (2n-(k+2))^{-3/2}\right) \cdot \left(\sum_{\ell=cn}^{2n-d \log n} \ell^{-3/2} (2n-\ell) \cdot (2n-(\ell+2))^{-3/2}\right).
    \]
    Each of the above sums will be $O(n^{-1})$ by the same argument as before. We conclude that the total contribution of these terms to the original sum is at most $n^3 \cdot O(n^{-1}) \cdot O(n^{-1}) = O(n) = o(n \log n)$, completing the proof. \qedhere
\end{itemize}
\end{proof}

We point out that using  Lemma~\ref{L-upperbound} and similar arguments to the ones used in cases 4 and 5 can be used to show that the region $d \log n \leq k,\ell \leq cn$ will contribute $O(n \log n)$ to the expected number of edges. In fact, using Lemma~\ref{L-GenLowProb} which we prove later on, we can also show a lower bound of $\Omega(n \log n)$ for this contribution. However, with a little bit more care it is possible to determine the exact constant.  We first require a probability lemma.

\begin{lem} \label{L-Chernoff}
Let $X_1, X_2, X_3, \ldots$ be independent random variables with $\Pr(X_i=0) = \Pr(X_i=1) = 1$, and set $S_j = \sum_{i=1}^j X_i$.
For $\epsilon > 0$, $d\ge 20/\epsilon^2$ and $j > d \log n$ we have
\[
P(|S_j - j/2| < \epsilon j/2) < 2 n^{-10}.
\]
\end{lem}

\begin{proof}
By \eqref{eq-Conc}, the desired probability is at most
\[
2 \exp(- 2 (\epsilon j/2)^2/j) = 2 \exp(- \epsilon^2 j/2) \leq 2 \exp(- \epsilon^2 d \log n/2) = 2 n^{-\epsilon^2 d/2} \leq 2 n^{-10}
\]
since $d \geq 20/\epsilon^2$.
\end{proof}

By Proposition~\ref{prop-edges}, in order to show \eqref{eq-edges} it suffices to prove that for suitably small $c$ and sufficiently large $d$ the contribution from good quadruples with $d \log n \leq k,\ell \leq cn$ is between
\[
(1-\epsilon) \frac1{\pi} n \log n \quad \textup{and} \quad (1+\epsilon) \frac1{\pi} n \log n.
\]
To this end we introduce the following notation, which intuitively means that two expression asymptotically gets arbitrarily close for $n \rightarrow \infty$, \emph{independent of all other variables}, provided one picks a suitably small $c$ and a suitably large $d$.

\begin{defn}
Let $f$ and $g$ be two functions with the same domain taking positive values, and whose inputs depend on some positive integer $n$ and some other integer variables, some of which are restricted to the interval $[d \log n,cn]$. We say that $f \simac g$ if for any $\epsilon > 0$ there exist suitable $c$, $d$ and $N$ with
\[
(1-\epsilon) f(x) \leq g(x) \leq (1+\epsilon) f(x)
\]
for any input $x$ with $n \geq N$.
\end{defn}

Here the subscript ac denotes that we do not have the exact asymptotic behavior, but that we get \emph{arbitrary close} asymptotic behavior by choosing suitable $c$ and $d$.

We now want to show that
\[
\sum_{(x,k,y,\ell)} \Pr[A(x,k,y,\ell)] \simac \frac1{\pi} n \log n
\]
where the sum is over all good quadruples $(x,k,y,\ell)$ with $d \log n \leq k,\ell \leq cn$. The desired result follows by the steps in the proposition below.

\begin{prop}\label{prop-edgeasym}
We have the following statements.
\begin{itemize}
\item[(i)] $\Pr[A(x,k,y,\ell)] \simac \frac1{16 \pi} k^{-3/2} \ell^{-3/2}$.
\item[(ii)] Let $g(k,\ell)$ be the number of pairs $(x,y)$ such that $(x,k,y,\ell)$ is a good quadruple. Then $g(k,\ell) \simac 4n \cdot \min\{k,\ell\}$.
\item[(iii)] We have
\[
\frac{n}{4\pi} \sum_{d \log n \leq k,\ell \leq cn} k^{-3/2} \ell^{-3/2} \cdot \min\{k,\ell\} \simac \frac{1}{\pi} n \log n.
\]
\end{itemize}
\end{prop}

Before proving this proposition, we first show that this implies the asymptotic result of Theorem~\ref{T-edges}.
\begin{cor}\label{C-edges1}
    The expected number of edges of $CP_n$ satisfies
    \[
        \E[e(CP_n)]\sim \rec{\pi}n\log n.
    \]
\end{cor}
\begin{proof}
	Given Proposition~\ref{prop-edgeasym}, for any $\epsilon>0$ there are some $c, d$ and $N$ such that for all $n \geq N$ we have \begin{align*}\Pr[A(x,k,y,\ell)] &\leq (1+\epsilon) \frac1{16 \pi} k^{-3/2} \ell^{-3/2}\\
	g(k,\ell) &\leq (1+\epsilon) 4n \min\{k,\ell\}\\
	\frac{n}{4\pi} \sum_{d \log n \leq k,\ell \leq cn} k^{-3/2} \ell^{-3/2} \cdot \min\{k,\ell\}  &\leq (1+\epsilon) \frac1{\pi} n \log n.\end{align*}  This implies
	
	\begin{align*}
	\sum_{(x,k,y,\ell)} \Pr[A(x,k,y,\ell)] &\leq (1+\epsilon) \sum_{(x,k,y,\ell)} \frac1{16 \pi} k^{-3/2} \ell^{-3/2} \\
	&= (1+\epsilon) \sum_{k,\ell} g(k,\ell) \cdot \frac1{16 \pi} k^{-3/2} \ell^{-3/2} \\
	&\leq (1+\epsilon)^2 \frac{n}{4\pi} \sum_{k,\ell} k^{-3/2} \ell^{-3/2} \min\{k,\ell\} \leq (1+\epsilon)^3 \frac1{\pi} n \log n,
	\end{align*}
	and similarly for the lower bound.
\end{proof}

We now prove this proposition.
\begin{proof}[Proof of Proposition~\ref{prop-edgeasym}]

\begin{itemize}
\item[(i)] It is clear that with probability $2^{-4}$ all of $x$, $x+k$, $y$, $y+\ell$ have the correct color. We now claim that, conditioning on the event that this happens, with probability $2^{-2}$ there is an even number of red points between $x$ and $x+k$ and an even number of blue points between $y$ and $y+\ell$. Indeed, consider the case where $x < y < x+k < y+\ell$.  Then for any possible coloring of $x+2, \ldots, y-1,y+1,\ldots,x+k-1,x+k+1,\ldots,y+\ell-2$ there is a unique choice of colors for $x+1$ and $y+\ell-1$ that makes the number of red and blue points in the respective regions even, and with probability $2^{-2}$ these points will receive this color (here we used our assumption that $y \geq x+2$ and $y+\ell \geq x+k+2$).

 Condition on the event that all of this happens.  Let $r_1$ and $r_2$ be defined such that there are $2r_1$ red dots between $x$ and $x+k$ and $2r_2$ red dots outside, and similarly define $b_1$ and $b_2$. Then, conditional on the aforementioned event, the probability of having arcs between $x$ and $x+k$ and $y$ and $y+\ell$ is given by
\[
\frac{C_{r_1} \cdot C_{r_2}}{C_{r_1+r_2+1}} \cdot \frac{C_{b_1} \cdot C_{b_2}}{C_{b_1+b_2+1}}.
\]

By Lemma~\ref{L-Chernoff}, with probability at least $1-8n^{-10}$ we have $r_1 \simac k/4$, $r_2 \simac n/2 - k/4$, $b_1 \simac \ell/4$ and $b_2 \simac n/2 - \ell/4$. Furthermore, since $k,\ell \geq d \log n$ and $d \log n \rightarrow \infty$ we may replace all Catalan numbers by their asymptotic expressions, which yields that the probability of having arcs on the desired positions is (asymptotically arbitrary closely) given by
\[
\frac{1}{16\pi} \cdot \left(\frac{r_1+r_2+1}{r_2}\right)^{3/2} r_1^{-3/2} \cdot \left(\frac{b_1+b_2+1}{b_2}\right)^{3/2} b_1^{-3/2}.
\]
Since $r_1+r_2+1 \simac n/2-k/4 + k/4 + 1 \simac n/2$ and $r_2 \simac n/2-k/4 \simac n/2$ (the latter since $n/2 \geq n/2-k/4 \geq n/2 - cn/4$), we find $\frac{r_1+r_2+1}{r_1} \simac 1$, and hence
\begin{align*}
\frac{1}{16\pi} \cdot \left(\frac{r_1+r_2+1}{r_2}\right)^{3/2} r_1^{-3/2} \cdot \left(\frac{b_1+b_2+1}{b_2}\right)^{3/2} b_1^{-3/2} &\simac \frac1{16 \pi} (k/4)^{-3/2} (\ell/4)^{-3/2}\\ &= 2^6 \frac1{16 \pi} k^{-3/2} \ell^{-3/2}.
\end{align*}
Therefore, for any $\epsilon$, and suitable $c$, $d$ and large enough $n$ we have
\begin{align*}
(1-8n^{-10})(1-\epsilon) \frac1{16 \pi} k^{-3/2} \ell^{-3/2} &\leq \Pr[A(x,k,y,\ell)] \\ &\leq (1-8n^{-10})(1+\epsilon) \frac1{16 \pi} k^{-3/2} \ell^{-3/2} + 8 n^{-10}.
\end{align*}
Since $1-8n^{-10} \rightarrow 1$ for $n \rightarrow \infty$, and since $k^{-3/2} \ell^{-3/2} \geq n^{-3}$ we have $n^{-10} = o(k^{-3/2} \ell^{-3/2})$ (uniformly in $n$). Hence this shows that $\Pr[A(x,k,y,\ell)] \simac \frac1{16\pi} k^{-3/2} \ell^{-3/2}$.
\item[(ii)] Without loss of generality we may assume that $k \leq \ell$. We show that $(4-6c)n(k-3) \leq g(k,\ell) \leq 4nk$. Since $k \geq d \log n$ and $d \log n \rightarrow \infty$ we have $k-3 \simac k$, and the result follows.

For the upper bound, note that we have at most $2n$ choices for $x$. Furthermore, given $x$, either $y$ or $y+\ell$ must be among $\{x+1,x+2,\ldots,x+k-1\}$, hence we have at most $2 \cdot (k-1) \leq 2k$ choices for $y$ afterwards. Therefore, $g(k,\ell) \leq 2n \cdot 2k = 4nk$.

For the upper bound, let $cn \leq x \leq (2-2c)n$. We claim that for any such $x$ there are at least $2(k-3)$ good quadruples with that $x$. Indeed, let $y \in \{x+2,\ldots,x+k-2\}$ or $y \in \{x+2-\ell,\ldots,x+k-2-\ell\}$, then we claim that any such $y$ satisfies. Since $\ell \geq k$ these two sets are disjoint, giving us $2(k-3)$ good quadruples.

First suppose that $y = x+j$ for $2 \leq j \leq k-2$. Then we clearly have $1 \leq x < y < x+k < y+\ell$, $y \geq x+2$ and $x+k \geq y+2$. Furthermore, $y+\ell \geq x+2+\ell \geq x+2+k = (x+k)+2$. Lastly, $y+\ell \leq x+k-2+\ell \leq 2n - 2cn + k+ \ell \leq 2n$, since $k,\ell \leq cn$. A similar argument holds in the case $y = x+j-\ell$.
\item[(iii)] We consider the contribution to the sum coming from $k < \ell$, the analysis for the contribution coming from $k \geq \ell$ is analogous. First, note that
\begin{align*}
\sum_{k < \ell} k^{-1/2} \ell^{-3/2} &= \sum_{\ell = d \log n}^{cn} \ell^{-3/2} \sum_{k  = d \log n}^{\ell-1} k^{-1/2} \leq \sum_{\ell} \ell^{-3/2} \int_1^\ell x^{-1/2} \mathrm{d}x \\ &=\sum_{\ell} \ell^{-3/2} (2 \ell^{1/2} - 2)
\leq \sum_{\ell=d \log n}^{cn} 2 \ell^{-1} \leq 2 \int_{d \log n-1}^{cn} x^{-1} \mathrm{d}x\\  &\leq 2 \log(cn) \leq 2 \log n.
\end{align*}
In the other direction, note that we have a lower bound of
\begin{align*}
\sum_{\ell = (\log n)^2}^{cn} \ell^{-3/2} \sum_{k = d \log n}^{\ell-1} k^{-1/2} &\geq \sum_{\ell = (\log n)^2}^{cn} \ell^{-3/2} \ell^{-3/2} \int_{d \log n}^{\ell} x^{-1/2} \mathrm{d}x\\  &= \sum_{\ell = (\log n)^2}^{cn} \ell^{-3/2} (2 \ell^{1/2} - 2 (d \log n)^{1/2}).
\end{align*}
For any $\epsilon$ we have $(d \log n) \leq \epsilon^2 (\log n)^2 \leq \epsilon^2 \ell^2$ for $n$ large enough, hence $2 \ell^{1/2} - 2 (d \log n)^{1/2} \geq 2(1-\epsilon) \ell^{-1/2}$ for $n$ large enough. Therefore, we get a lower bound of
\[
2(1-\epsilon) \sum_{\ell = \log(n)^2}^{cn} \ell^{-1} \geq 2(1-\epsilon) \left( \log(cn+1) - \log((\log n)^2) \right)
\]
by again comparing the sum with an integral. The desired result now follows from the fact that
\[
\log(cn+1) - \log(\log(n)^2) \geq \log n + \log c - \log(\log(n)^2) \sim \log n,
\]
hence we have $\log(cn+1) - \log(\log(n)^2) \geq (1-\epsilon) \log n$ for $n$ large enough. \qedhere
\end{itemize}
\end{proof}

\section{The number of isolated vertices} \label{sec-isolated}
In this section we will determine the asymptotic behavior of the number of isolated vertices, as stated in Theorem~\ref{T-isolated}. Recall that $I_n$ denotes the number of isolated vertices of $CP_n$ and that we defined \[
\gamma = 4 \sum_{m=1}^{\infty} 16^{-m} \sum_{b=0}^{m-1} \binom{2m-2}{2b} C_{m-1-b} C_b.
\]

Before proving Theorem~\ref{T-isolated}, let us first show why the sum defining $\gamma$ is a convergent sum. Let $\gamma_m = 4 \cdot 16^{-m} \sum_{b=0}^{m-1} \binom{2m-2}{2b} C_{m-1-b} C_b$, then as noted in \cite[Section 5]{Catalan-pair} we have $\gamma_m \leq \frac1{4(m-1)^2}$ for $m \geq 2$, from which the convergence follows since the sum of the reciprocals of the squares converges. In fact, this gives us an error bound on how quickly the finite sums $\sum_{m=1}^M \gamma_m$ converge to $\gamma$. Indeed
\begin{align*}
\gamma = \sum_{m=1}^{\infty} \gamma_m &= \sum_{m=1}^M \gamma_m + \sum_{m=M+1}^{\infty} \gamma_m \leq \sum_{m=1}^M \gamma_m + \sum_{m=M+1}^{\infty} \frac1{4(m-1)^2} \\
    &\leq \sum_{m=1}^M \gamma_m + \int_{x = m}^{\infty} \frac1{4(x-1)^2} \mathrm{d}x = \sum_{m=1}^M \gamma_m + \frac1{4(M-1)}.
\end{align*}
Using the trivial lower bound $\gamma \geq \sum_{m=1}^M \gamma_m$ and taking $M = 10^4$ one can compute that
\[
0.30234 \leq \gamma \leq 0.30238.
\]

We first show that $\E[I_n]$ is asymptotically at least $\gamma n$. As a first observation we note that any arc yielding an isolated vertex must have an even number of points between its endpoints, as otherwise there would be an arc connecting a point between its endpoints with a point outside. Such an arc would necessarily be on the other side and would yield an edge involving the arc in question. Therefore, $I_n = \sum_{m=1}^n I_{n,m}$ where $I_{n,m}$ is the number of isolated vertices induced by an arc connecting two points with $2m-2$ points between them.

The following result will suffice to prove the lower bound for $\E[I_n]$.

\begin{prop} \label{prop-isolated-m}
For $m$ a fixed positive integer we have $\E[I_{n,m}] \sim \gamma_m n$.
\end{prop}

As a result of this proposition, we can see that
\[
\E[I_n] \geq \sum_{m=1}^M \E[I_{n,m}] \sim \sum_{m=1}^M \gamma_m n,
\]
which gets arbitrarily close (in the multiplicative sense) to $\gamma n$ by picking $M$ large enough. However, this approach does not immediately yield the upper bound, since each $E[I_{n,m}]$ will converge to $\gamma_m n$ at its own rate, hence a bit more care is needed to handle the full sum $\E[I_n] = \sum_{m=1}^n \E[I_{n,m}]$.

\begin{proof}[Proof of Proposition~\ref{prop-isolated-m}]
We count the expected number of such arcs that come from the top, and by symmetry we can multiply this quantity by two to get our final answer.  As mentioned above, an arc connecting $x$ and $x+2m-1$ is isolated if and only if the $2m-2$ intermediate points are only connected to themselves. The total number of ways to connect those points is given by
\[
\sum_{b=0}^{m-1} \binom{2m-2}{2b} C_{m-1-b} C_b,
\]
where $b$ is the number of arcs on the bottom, $\binom{2m-2}{2b}$ counts the number of ways to select the $2b$ points for these arcs, and $C_{m-1-b}$ and $C_b$ count the number of ways to choose the arcs on the top and the bottom.

Now fix one such configuration with $b$ arcs on the bottom and $a$ arcs on top (including the arc between $x$ and $x+2m-1$). We claim that the expected number of such configurations in $CP_n$ is given by
\[
(2n-2m+1) 2^{-2m} \sum_{r=0}^{n-m} p_r \frac{C_r}{C_{r+a}} \cdot \frac{C_{n-m-r}}{C_{n-m-r+b}},
\]
where $p_r = p_r(n,a,b)$ is the probability that $2r$ of the points not among the $2m$ specified points are colored red.

This formula follows from the fact that there are $2n-2m+1$ possibilities for $x$, namely $1 \leq x \leq 2n-2m+1$, and that for each such $x$ the probability of the points $x, x+1, \ldots, x+2m-1$ colored exactly as in our configuration is given by $2^{-2m}$. After that, given $x$ and conditioning on these points having the correct colors and conditioning on there being $2r$ other red points, the probability that the top Catalan-arc matching (which has size $r+a$) has exactly the desired configuration on our given $2a$ red points is exactly $\frac{C_r}{C_{r+a}}$ by Lemma~\ref{L-CatProb}, and a similar result holds for the probability of the bottom Catalan-arc matching coinciding with our given configuration on the $2b$ points.

To complete the proof it suffices to show that
\[
\sum_{r=0}^{n-m} p_r \frac{C_r}{C_{r+a}} \cdot \frac{C_{n-m-r}}{C_{n-m-r+b}} \sim 4^{-m},
\]
since then
\[
\E[I_{n,m}] \sim 2 \left(\sum_{b=0}^{m-1} \binom{2m-2}{2b} C_{m-1-b} C_b\right) (2n-2m+1) 2^{-2m} \cdot 4^{-m} \sim \gamma_m n.
\]
Using \eqref{eq-Conc} with exponential small probability we have $r \leq n/4$ or $n-m-r \leq n/4$. As a trivial lower bound we have
\[
\sum_{r=0}^{n-m} p_r \frac{C_r}{C_{r+a}} \cdot \frac{C_{n-m-r}}{C_{n-m-r+b}} \geq \sum_{r=n/4}^{n-m-n/4} p_r \frac{C_r}{C_{r+a}} \cdot \frac{C_{n-m-r}}{C_{n-m-r+b}}.
\]
Now in this region, since $r, r+a, n-m-r, n-m-r+b \geq n/4$ we can use the approximation for the Catalan numbers from \eqref{eq-CatAsy} and find the lower bound
\begin{align*}
\sum_{r=n/4}^{n-m-n/4} p_r \frac{C_r}{C_{r+a}} \cdot \frac{C_{n-m-r}}{C_{n-m-r+b}}
&\sim \sum_{r=n/4}^{n-m-n/4} p_r \frac{4^r}{4^{r+a}} \left(\frac{r+a}{r}\right)^{3/2} \cdot \frac{4^{n-m-r}}{4^{n-m-r+b}} \left(\frac{n-m-r}{n-m-r+b}\right)^{3/2} \\
&\sim \sum_{r=n/4}^{n-m-n/4} p_r 4^{-(a+b)} = 4^{-m} \sum_{r=n/4}^{n-m-n/4} p_r \sim 4^{-m},
\end{align*}
where the last step follows from the fact that $r < n/4$ or $r > n-m-n/4$ holds with exponentially small probability.

Similarly, we have
\begin{align*}
\sum_{r=0}^{n-m} p_r \frac{C_r}{C_{r+a}} \cdot \frac{C_{n-m-r}}{C_{n-m-r+b}} &\leq \sum_{r=n/4}^{n-m-n/4} p_r \frac{C_r}{C_{r+a}} \cdot \frac{C_{n-m-r}}{C_{n-m-r+b}} + \Pr(r \leq n/4 \textup{ or } n-m-r \leq n/4)\\
    &\sim 4^{-m} + \Pr(r \leq n/4 \textup{ or } n-m-r \leq n/4) \sim 4^{-m},
\end{align*}
completing the proof.
\end{proof}

We now prove the desired asymptotics for the number of isolated vertices.

\begin{prop}\label{P-isolatedasy}
Let $\gamma$ be the constant defined by
\[
\gamma = 4 \sum_{m=1}^{\infty} 16^{-m} \sum_{b=0}^{m-1} \binom{2m-2}{2b} C_{m-1-b} C_b  = 0.3023\ldots.
\]
Let $I_n$ denote the number of isolated vertices of $CP_n$.  Then $\E[I_n] \sim \gamma n$.
\end{prop}
\begin{proof}
As mentioned after the statement of Proposition~\ref{prop-isolated-m} we have shown an asymptotic lower bound of $\gamma n$ on the number of isolated vertices.  For the upper bound, note that using the notation of Lemma~\ref{L-Range} we have that $I_{n,m}\le A_{2m-1,2m-1}$, since the number of isolated vertices coming from arcs of length $2m-1$ is clearly at most the the total number of arcs of this length. By this observation, the fact that $\sum_{m=16 \log n+1}^n A_{2m-1,2m-1} \leq A_{32 \log n+1,2n}$, and Lemma~\ref{L-Range}, we have
\[
\sum_{m=16\log n+1}^{n} \E[I_{n,m}]\le \E[A_{32\log n+1,2n}]=o(n),
\]
which shows that
\[
\E[I_n] =
\sum_{m=1}^{16 \log n} \E[I_{n,m}] + o(n).
\]
Using the argument from Proposition~\ref{prop-isolated-m} we see that
\[
\sum_{m=1}^{16 \log n} \E[I_{n,m}] \leq 4n \sum_{m=1}^{16 \log n} 4^{-m} \sum_{b=0}^{m-1} \binom{2m-2}{2b}C_{m-1-b}C_b \sum_{r=0}^{n-m} p_r(n,a,b) \frac{C_r}{C_{r+a}} \cdot \frac{C_{n-m-r}}{C_{n-m-r+b}}.
\]
We now see that for any $m$, $a$ and $b$ we have that there are at least $n$ points outside of the configuration, hence $2r$ is the sum of at least $n$ independent $0-1$ Bernoulli $p=1/2$ variables. This means that with at most some exponentially small probability $c^{-n}$ we have $r, n-m-r \leq n/10$.

Therefore, for all cases where $r, n-m-r \geq n/10$ we can again (uniformly over all summands) replace $\frac{C_r}{C_{r+a}}$ by $4^{-a} \left(\frac{r+a}{r}\right)^{3/2}$. Since $\frac{r+a}{r} = 1 + \frac{a}{r} \leq 1 + \frac{16 \log n}{n/10}$ we can asymptotically replace $\frac{r+a}{r}$ by $1$ over all summands. Using this and the approach as in Proposition~\ref{prop-isolated-m} we have an asymptotic upper bound $\sum_{r=0}^{n-m} p_r(n,a,b) \frac{C_r}{C_{r+a}} \cdot \frac{C_{n-m-r}}{C_{n-m-r+b}} \leq 4^{-m} + c^{-n}$, hence (asymptotically up to arbitrarily small multiplicative factors) we have
\begin{align*}
\sum_{m=1}^{16 \log n} \E[I_{n,m}] &\leq 4n \sum_{m=1}^{16 \log n} 4^{-m} \sum_{b=0}^{m-1} \binom{2m-2}{2b}C_{m-1-b}C_b \left(4^{-m} + c^{-n}\right) \\
 &\leq \gamma n + 4n \left(\sum_{m=1}^{16 \log n} 4^{-m} \sum_{b=0}^{m-1} \binom{2m-2}{2b} C_{m-1-b} C_b \right) c^{-n} \\
 &\leq \gamma n + 4n \left(\sum_{m=1}^{16 \log n} 4^{-m} 16^m\right) c^{-n} \leq \gamma n + 4n c^{-n} \sum_{m=1}^{16 \log n} 4^m \\
 &\leq \gamma n + 4n c^{-n} \cdot 16 \log n 4^{16 \log n} = \gamma n + 64 n c^{-n} \cdot \log n \cdot n^{16 \log 4} = \gamma n + o(1),
\end{align*}
since $c^{-n}$ goes to zero faster than $n^{1 + 16 \log 4} \log n$ grows to infinity.
\end{proof}

We can use a similar proof to bound the variance of $I_n$.

\begin{prop} \label{prop-isolated-variance}
The variance of the number of isolated vertices in $CP_n$ satisfies $\Var[I_n] = o(n^2)$.
\end{prop}

Before giving this proof, let us point out that using Chebyshev's inequality we can use this result to complete the proof of Theorem~\ref{T-isolated}.
\begin{proof}[Proof of Theorem~\ref{T-isolated}]
The asymptotic result for the expected number of isolated vertices follows from Proposition~\ref{P-isolatedasy}.  From this we know that $|\E[I_n] - \gamma n| < \epsilon/2 \cdot n$ for $n$ large enough. Hence, for sufficiently large $n$ we have,
\[
\Pr[|I_n - \gamma_n| > \epsilon n] \leq \Pr[|I_n - \E[I_n]| > \epsilon/2 \cdot n].
\]
Now, applying Chebyshev's inequality we find
\[
\Pr[|I_n - \E[I_n]| > \epsilon/2 \cdot n] \leq \frac{\Var[I_n]}{(\epsilon/2 \cdot n)^2} = \frac{o(n^2)}{(\epsilon/2 \cdot n)^2} = o(1),
\]
as desired.
\end{proof}

We will now prove the result on the variance.
\begin{proof}[Proof of Proposition~\ref{prop-isolated-variance}]
By definition we have $\Var[I_n] = \E[I_n^2] - \E[I_n]^2$, where $\E[I_n]^2 = (\gamma n)^2 + o(n^2)$ by the first part of Theorem~\ref{T-isolated}. Therefore, since variance is nonnegative, it suffices to show that
\[
\E[I_n^2] \leq (\gamma n)^2 + o(n^2).
\]
Observe that $I_n^2$ is the number of ordered pairs of isolated vertices.

Just as above we show that we can restrict ourselves to the isolated vertices induced by arcs of length at most $32 \log n$. Indeed, let $A_{\al,\be}$ be as in Lemma~\ref{L-Range}.  Then the number of pairs where at least one vertex comes from an arc of length at least $32 \log n$ is at most $2 \cdot A_{32\log n,2n}\cdot n$, where the factor $2$ represents the choice of the vertex coming from a long arc being the first or second vertex in the pair, $A_{32\log n,2n}$ is the number of ways to pick this long arc, and $n$ is the number of ways to pick the remaining vertex. Therefore, this contribution to $\E[I_n^2]$ is at most $\E[2 \cdot A_{32\log n,2n}\cdot n] = o(n^2)$ by  Lemma~\ref{L-Range}.

Additionally, the number of pairs of isolated vertices coming from two arcs of length at most $32 \log n$, where one arc is contained in the other arc (possibly facing the other way) is deterministically at most $O(n \log n)$, since one can pick the outer arc in at most $n$ ways and then there are at most $32 \log n$ ways to pick the smaller arc. Therefore, these pairs contribute $o(n^2)$ to $\E[I_n^2]$ as well. Furthermore, the number of pairs where both arcs are the same are at most $n$, so these will also contribute $o(n^2)$ to $\E[I_n^2]$.

Therefore, we can restrict our attention to pairs of isolated vertices coming from different arcs of length at most $32 \log n$ such that neither arc is contained in the other. Note that since the arcs yield isolated vertices their endpoints cannot interlace, so the sets of points covered by this arc are disjoint.

Suppose we want to calculate the probability of having a pair of isolated vertices, one of them induced by an arc connecting $(x,x+2m-1)$ and the other connecting an arc connecting $(y,y+2k-1)$, where $m,k \leq 16 \log n$. By a similar argument as in Proposition~\ref{prop-isolated-m}, after specifying configurations for $\{x+1,\ldots,x+2m-2\}$ and $\{y+1,\ldots,y+2k-2\}$ the probability is (asymptotically up to arbitrarily small multiplicative factors) at most
\[
4^{-(m+k)} \cdot (4^{-(m+k)} + c^{-n}),
\]
where $4^{-(m+k)}$ is the probability that all of $\{x,x+1,\ldots,x+2m-1\}$ and $\{y,y+1,\ldots,y+2k-1\}$ receive the correct color, and $c^{-n}$ is once again an upper bound on the probability of \emph{not} having at least $n/10$ more blue and red points, and the $4^{-(m+k)}$ is once again the factor that shows up by considering the asymptotic behavior of the appropriate quotient of Catalan numbers. Also, by the same argument we can do these asymptotics for all possible $x$, $y$, $k$, $m$ and choice of configurations simultaneously.

Taking into account that there are at most $(2n)^2$ ways to choose $x$ and $y$, and $4$ ways to choose the side (top or bottom) for the arcs, and considering the possible configurations for $\{x+1,\ldots,x+2k-2\}$ and $\{y+1,\ldots,y+2k-2\}$ we find an asymptotic upper bound for the desired contribution of
\[
\sum_{k,m=1}^{16 \log n} 16 n^2 \left(\sum_{b_1=0}^{m-1} \binom{2m-2}{2b_1} C_{m-1-b_1} C_{b_1}\right)\left(\sum_{b_2=0}^{k-1} \binom{2k-2}{2b_2} C_{k-1-b_2}C_{b_2}\right) 4^{-(m+k)} \left(4^{-(m+k)} + c^{-n}\right).
\]
Using $4^{-(m+k)} + c^{-n} \leq (4^{-m} + c^{-n/2})(4^{-k}+c^{-n/2})$, we can separate the sums over $k$ and $m$.  Thus the contribution is at most
\[
\left(\sum_{m=1}^{16 \log n} 4n \cdot \sum_{b_1=0}^{m-1} \binom{2m-2}{2b_1} C_{m-1-b_1} C_{b_1} \cdot 4^{-m} (4^{-m} + c^{-n/2})\right)^2 \leq (\gamma n + o(1))^2 = (\gamma n)^2 + o(n^2),
\]
where the last inequality once again follows from the proof of Theorem~\ref{T-isolated}.
\end{proof}

We note that essentially the same proof can be used to show that $\E[I_n^m] \sim \gamma^m n^m$ for all $m \geq 2$.

\section{The variance of the number of edges}
\label{sec-edgevariance}
This section will be devoted to bounding the variance of the random variable $e(CP_n)$. We will prove the following result, which with a proof similar to that of Theorem~\ref{T-isolated} will imply the concentration result of Theorem~\ref{T-edges}.
\begin{prop} \label{prop-edgevariance}
    The variance of the number of edges in $CP_n$ satisfies
    \[
    \Var[e(CP_n)] = o(n^2 \log^2 n).
    \]
\end{prop}

Similar to the case of isolated vertices, we will prove this statement by showing that for any $\epsilon > 0$ and $n$ large enough we have
\[
E[(e(CP_n))^2] \leq (1+\epsilon) \frac1{\pi^2} n^2 \log^2 n + o(n^2 \log^2 n).
\]
In other words, we want to count the expected number of pairs of edges in $CP_n$. Just as when we determined the expected number of edges, we first have to handle some exceptional cases and show that all of these cases contribute of order $o(n^2 \log^2 n)$. This requires a few more cases than before, and each of the proofs will be a bit longer since there are more things to take care of. Since the general approach of all of the proofs are similar to Proposition~\ref{prop-edges} and Proposition~\ref{prop-isolated-variance}, we will only state the lemmas here and defer the proofs to Appendix~\ref{appendix-edgevariance}.

As mentioned, $e(CP_n)^2$ is the number of pairs of edges in $CP_n$. Typically, such a pair of edges will be induced by four arcs in the representative for $CP_n$. The first step will be to show that these pairs are indeed the main contribution to $E[(e(CP_n))^2]$.

\begin{lem} \label{L-variancefewarcs}
    The expected number of pairs of edges in $CP_n$ induced by at most three arcs in its representative is at most $o(n^2 \log^2 n)$.
\end{lem}

Therefore, we can restrict to valid quadruples $q = (\b{x},\b{k},\b{y},\b{l}) = ((x_1,x_2),(k_1,k_2),(y_1,y_2),(\ell_1,\ell_2))$ where $(x_i,k_i,y_i,\ell_i)$ is a possible edge for $i = 1,2$. Our goal is now to show that
\[
\sum_q \Pr[A(\b{x},\b{k},\b{y},\b{l})] \leq (1+\epsilon) \frac1{\pi^2} n^2 \log^2 n + o(n^2 \log^2 n),
\]
where the sum is over all valid quadruples $q = (\b{x},\b{k},\b{y},\b{l})$. We use the notation for $f_0, f_1, f_2, g_0, g_1, g_2$ as in Section~\ref{sec-edges}. Similar to the proof for the expected number of edges, the first step will be to show that the main contribution comes from quadruples with $f_i, g_j \geq d \log n$. That is we will show that if $Q_1$ is the set of quadruples for which at least one of $f_i$, $g_j$ is less than $d \log n$, then
\[
\sum_{q \in Q_1} \Pr[A(\b{x},\b{k},\b{y},\b{l})] = o(n^2 \log^2 n).
\]
Without loss of generality we can consider the case where one of the $f_i$ is less than $d \log n$. Then the result follows from the two lemmas below, the first one of which deals with the case that the two arcs on top are nested, and the second one deals with the unnested case.

\begin{lem} \label{L-smallnested}
    Let $Q_{1,1}$ be the set of all valid quadruples $q$ for which $x_1 < x_2 < x_2 + k_2 < x_1 + k_1$ and for which $k_2, k_1 - k_2$ or $2n - k_1$ is less than $d \log n$. Then
    \[
    \sum_{q \in Q_{1,1}} \Pr[A(\b{x},\b{k},\b{y},\b{l})] = o(n^2 \log^2 n).
    \]
\end{lem}

\begin{lem} \label{L-smallunnested}
    Let $Q_{1,2}$ be the set of all valid quadruples $q$ for which neither $x_1 < x_2 < x_2 + k_2 < x_1 + k_1$ nor $x_2 < x_1 < x_1 + k_1 < x_2 + k_2$ holds, and for which $k_1, k_2$ or $2n - (k_1+k_2)$ is less than $d \log n$. Then
    \[
    \sum_{q \in Q_{1,2}} \Pr[A(\b{x},\b{k},\b{y},\b{l})] = o(n^2 \log^2 n).
    \]
\end{lem}

In order to complete the proof of Proposition~\ref{prop-edgevariance} we can now assume that all $f_i, g_j$ are at least $d \log n$. The first step will be to deal with the case that some of the arcs are nested.

\begin{lem} \label{L-variancenested}
    Let $Q_2$ be the set of quadruples with $x_1 < x_2 < x_2 + k_2 < x_1 + k_1$ and $f_i, g_j \geq d \log n$. Then
    \[
    \sum_{q \in Q_2} \Pr[A(\b{x},\b{k},\b{y},\b{l})] = o(n^2 \log^2 n)
    \]
\end{lem}

For the remainder of this section on we will assume that any quadruple has no nested arcs. First we take care of the quadruples where one of the arcs is too large.

\begin{lem} \label{L-variancelong}
Let $Q_3$ be the set of quadruples with $\max\{k_1,k_2,\ell_1,\ell_2\} > cn$. Then
    \[
    \sum_{q \in Q_3} \Pr[A(\b{x},\b{k},\b{y},\b{l})] = o(n^2 \log^2 n)
    \]
\end{lem}

We lastly rule out all of the remaining quadruples that are valid but not good.

\begin{lem} \label{L-variancenotgood}
Let $Q_4$ be the set of valid quadruples that are not good and have $d \log n \leq k_1,k_2,\ell_1,\ell_2 \leq cn$. Then
    \[
    \sum_{q \in Q_4} \Pr[A(\b{x},\b{k},\b{y},\b{l})] = o(n^2 \log^2 n).
    \]
\end{lem}

Before we give the proof of Proposition~\ref{prop-edgevariance} we recall a definition from Proposition~\ref{prop-edgeasym}. For positive integers $k,\ell$, we defined $g(k,\ell)$ as the number of pairs $(x,y)$ such that $(x,k,y,\ell)$ is a good quadruple. We are now ready to prove our desired result on the variance.
\begin{proof}[Proof of Proposition~\ref{prop-edgevariance}]
By Lemmas~\ref{L-variancefewarcs} through \ref{L-variancenotgood} we only have to consider quadruples $(\b{x},\b{k},\b{y},\b{l})$ that are good, have no nested arcs, and which have $d \log n \leq k_1,k_2,\ell_1,\ell_2 \leq cn$. In this case, given $k_1,k_2,\ell_1,\ell_2$ there are $g(k_1,\ell_1)$ ways to pick $x_1, y_1$ and after that at most $g(k_2,\ell_2)$ ways to pick $x_2, y_2$.

Therefore, it suffices to show that for $d$ large enough and $c$ small enough we have
\begin{equation} \label{eq-edgevariance}
\Pr[A(\b{x},\b{k},\b{y},\b{l})] \leq (1+\epsilon) \cdot \frac1{16 \pi} k_1^{-3/2} \ell_1^{-3/2} \cdot \frac1{16 \pi} k_2^{-3/2} \cdot \ell_2^{-3/2},
\end{equation}
as this implied that the desired contribution is at most
\[
\sum_{k_1,k_2,\ell_1,\ell_2} g(k_1,\ell_1) \cdot g(k_2,\ell_2) \cdot (1+\epsilon) \cdot \frac1{16 \pi} k_1^{-3/2} \ell_1^{-3/2} \cdot \frac1{16 \pi} k_2^{-3/2} \cdot \ell_2^{-3/2},
\]
which factors as
\[
(1+\epsilon) \left(\sum_{k_1,\ell_1} g(k_1,\ell_1) \frac1{16 \pi} k_1^{-3/2} \ell_1^{-3/2}\right) \cdot \left(\sum_{k_2,\ell_2} g(k_2,\ell_2) \frac1{16 \pi} k_2^{-3/2} \ell_2^{-3/2}\right),
\]
which by Proposition~\ref{prop-edgeasym} is at most $(1+\epsilon)^3 \cdot (\frac1{\pi} n \log n)^2$ for $d$ large enough and $c$ small enough.

In order to show \eqref{eq-edgevariance} we follow the same approach as the proof of part 1 of Proposition~\ref{prop-edgeasym}. First, with probability $2^{-8}$ all of $x_i, x_i+k_i, y_i, y_i+\ell_i$ receive the correct color and with probability $2^{-4}$ the number of red points between $x_i$ and $x_i+k_i$ and the number of blue points between $y_j$ and $y_j+\ell_j$ are all even. This follows immediately from the aforementioned proof when neither $(x_1,x_1+k_1)$ and $(y_2,y_2+\ell_2)$ nor $(x_2,x_2+k_2)$ and $(y_1,y_1+\ell_1)$ intersect. Otherwise, we may without loss of generality assume that $x_1 < y_1 < x_1+k_1 < x_2 < y_1+\ell_1 < y_2 < x_2+k_2 < y_2+\ell_2$. In this case, color all the remaining points between $x_1$ and $y_2+\ell_2$ except for $x_1+1, y_1+1, x_2+1, y_2+1$. Then, given any such coloring there is a unique choice for the remaining four colors that makes the number of red/blue in the desired regions even, as first $y_2+1$ is uniquely determined, then $x_2+1$, then $y_1+1$ and lastly $x_1+1$.

Now suppose that $r_i$ is half the number of red points between $x_i$ and $x_i+k_i$ for $i=1,2$, $r_0$ is half the number of red points outside of the arcs, and $b_0, b_1, b_2$ are defined similarly. Conditioned on the values of $r_i$ and $b_j$ we can write the desired probability as
\[
\frac{C_{r_0} C_{r_1} C_{r_2}}{C_{r_0+r_1+r_2+2}} \cdot \frac{C_{b_0} C_{b_1} C_{b_2}}{C_{b_0+b_1+b_2+2}}.
\]
Again by Lemma~\ref{L-Chernoff}, with high enough probability we can approximate $r_i$ with $k_i/4$ ($i=1,2$) and $r_0$ with $n/4 - k_1/4 - k_2/4$, and similarly for the $b_i$, and the same asymptotic considerations as in Proposition~\ref{prop-edgeasym} will now yield the desired result.
\end{proof}

With all this we can conclude the results of Theorem~\ref{T-edges}.
\begin{proof}[Proof of Theorem~\ref{T-edges}]
The asymptotic formula for the expected number of edges follows from Corollary~\ref{C-edges1}.  The concentration result follows from Proposition~\ref{prop-edgevariance} and essentially the same proof used in the proof of Theorem~\ref{T-isolated}.
\end{proof}

\section{Induced subgraphs and connected components} \label{sec-induced}
In this section we prove results on the number of induced subgraphs of $CP_n$ isomorphic to a given Catalan-pair graph $H$ on at least $3$ vertices, and we will use this to prove Theorem~\ref{T-induced}.  At the end of the section we will also discuss a result about the connected components of $CP_n$.

\subsection{A lower bound for the number of induced subgraphs} \label{subsec-lowerbound}

Recall that $N^*_H(G)$ denotes the number of induced subgraphs of $G$ isomorphic to $H$, and  that $A(\b{x},\b{k},\b{y},\b{l})$ denotes the intersection of the following events.
\begin{itemize}
\item[1.] The points $x_i$ and $x_i+k_i$ are colored red and the points $y_j$ and $y_j+\ell_j$ are colored blue for all $i,j$.
\item[2.] For all $i$ and $j$ the number of red points $x$ with $x_i<x<x_i+k_i$ and the number of blue points $y$ with $y_j<y<y_j+\ell_j$ is even.
\item[3.] For all $i$ and $j$ we have that $(x_i,x_i+k_i)$ and $(y_j,y_j+\ell_j)$ match in $CP_n$.
\end{itemize}

The following lemma will be a key step to proving the general lower bound. Note that this lemma can be seen as a converse to Lemma~\ref{L-upperbound}.

\begin{lem}\label{L-GenLowProb}
There exists a positive real number $\alpha_{s,t}$ with
\[
\Pr[A(\b{x},\b{k},\b{y},\b{l})] \ge \alpha_{s,t} \prod_{i=1}^s k_i^{-3/2} \prod_{j=1}^t \ell_j^{-3/2}
\]
for all good quadruples $(\bf{x},\bf{k},\bf{y},\bf{l})$ where $\bf{x}$ and $\bf{y}$ have length $s$ and $t$ respectively.
\end{lem}

\begin{proof}
We first show that with probability $2^{-3 (s+t)}$ the first two conditions are satisfied. It is clear that with probability $1/2$ all of the points $x_i, x_i+k_i, y_j, y_j+\ell_j$ receive the correct color, so with probability $2^{-2(s+t)}$ all of these points have the correct color. Now conditioned on all of these points having the correct color, we show that with probability $2^{-(s+t)}$ the second condition is satisfied. Consider all the points of the form $x_i+1$ and $y_j+1$, and note that by assumption of $(\b{x},\b{k},\b{y},\b{l})$ being a good quadruple all of these points are different and not equal to any of the $x_i, x_i+k_i, y_j$ and $y_j+\ell_j$. Consider the rightmost of these points, and suppose that it is equal to $x_i+1$ for some $i$. Since all of the points to the right have been colored, we have that in particular all of the points $x$ with $x_i < x < x_i+k_i$ except for this one have been colored. Therefore there is a unique choice for the color of $x_i+1$ that makes the number of red points $x$ with $x_i < x < x_i+k_i$ even. Inductively apply this argument for the remaining points, always taking the rightmost uncolored point.

Now suppose the first two conditions are satisfied. We apply Lemma~\ref{L-CatProb} to determine a lower bound for the probability that the third condition is met. To this end, for each $1 \leq i \leq s$ let $2r_i$ be the number of red points $x$ with $x_i < x < x_i+k_i$ that do not satisfy $x_j \leq x \leq x_j + k_j$ for any $j \neq i$, and let $2r_0$ be the number of red points that have not been counted for any of the $r_i$ and is not of the form $x_i$ or $x_i+k_i$. Define $b_0, b_1, \ldots, b_t$ similarly. Let $R_n$ and $B_n$ denote the total number of red and blue points respectively. Note that for any $1 \leq i \leq s$ we have $2r_i \leq k_i$, hence in particular $r_i \leq k_i$. Now applying the aforementioned lemma we find that
\begin{align*}
\Pr[A(\b{x},\b{k},\b{y},\b{l})] &\geq 2^{-3(s+t)} \cdot \alpha_s \sideset{}{'}\prod \frac{R_n^{3/2}}{r_i^{3/2}} \cdot \alpha_t \sideset{}{'}\prod \frac{B_n^{3/2}}{b_j^{3/2}} \\
    &\geq \alpha_{s,t} \prod_{i=0}^s \frac{R_n^{3/2}}{\max(r_i,1)^{3/2}} \cdot \prod_{j=0}^t \frac{B_n^{3/2}}{\max(b_j,1)^{3/2}} \\
    &\geq \alpha_{s,t} \prod_{i=1}^s k_i^{-3/2} \prod_{j=1}^t \ell_j^{-3/2},
\end{align*}
where we used that $R_n \geq \max(r_0,1)$, $B_n \geq \max(b_0,1)$, $\max(r_i,1) \leq k_i$ and $\max(b_j,1) \leq \ell_j$.
\end{proof}

We are now ready to prove the lower bound of Theorem~\ref{T-induced}.  In fact, we will give a lower bound for any Catlan-pair graph regardless of whether it is connected or not.

\begin{prop}\label{P-inducedLower}
		Let $H$ be a Catalan-pair graph on $v$ vertices with $i$ isolated vertices and $m$ isolated edges. Then
		\[
		\E[N^*_H(CP_n)]=\Om(n^{\f{v+i}{2}}(\log n)^m).
		\]
\end{prop}

	\begin{proof}
		We will prove this by first showing that the result holds for $m=i=0$, then for $i=0$, and finally for arbitrary $m$ and $i$.  We note that one can prove the most general case without first going through the other two cases, but this would decrease the readability of the proof.
		
		First assume $m=i=0$, and let $q_H$ be any quadruple representing $H$.  Our goal will be to find a large number of ``blowups'' of $q_H$.  Let $c\ge 4v$ be a fixed constant, and let \[P_j:=\{1+(j-1)\floor{n/c},2+(j-1)\floor{n/c},\ldots,-1+j\floor{n/c}\},\ \]\[P:=P_1\times \cdots \times P_{2v}.\]  Given $p=(p_1,\ldots,p_{2v})\in P$, we will define a quadruple $q_c(p)$ as follows.  If in $q_H$ we have $x_j=a$ and $x_j+k_j=b$, then in $q_c(p)$ we let $x_j=p_a$ and $x_j+k_j=p_b$, and we similarly define $y_j$ and $y_j+\ell_j$ to correspond to the bottom $j$th arc of $q_H$. We note that the reason we force all the points of the left of $2v \lfloor n/c \rfloor \leq n/2$ is to make sure that in the general case we have enough space left to place or find arcs yielding the isolated edges and vertices.
		
		We claim that $q_c(p)$ is a good quadruple that represents $H$ for any $p\in P$.  First observe that the points of $q_c(p)$ have the same relative order as the points of $q_H$, which shows that $q_c(p)$ satisfies the third condition for being a good quadruple (since $q_H$ satisfies this condition), and moreover that $q_c(p)$ represents $H$.  The first condition for being a good quadruple follows since the largest point we could choose for $q_c(p)$ is $-1+2v\floor{n/c}\le n/2$ since $c\ge 4v$, and the second condition follows since $|\max P_j-\min P_k|\ge 2$ for all $j,k$ by the way we defined these sets.  This proves our claim.
		
		Now let $Q_H(c)$ denote the set of all $q_c(p)$ with $p\in P$.  Observe that \[|Q_H(c)|= (\floor{n/c}-1)^{2v}\ge (2c)^{-2v} n^{2v}\] for $n$ sufficiently large.  Also observe that since $k_j,\ell_j\le 2n$ for all $j$,  Lemma~\ref{L-GenLowProb} gives that $\Pr[A(\b{x},\b{k},\b{y},\b{l})]\ge \al_v n^{-3v/2}$ for all $(\b{x},\b{k},\b{y},\b{l})\in Q_H(c)$, where $\al_v:=2^{-3v/2}\max_{s+t=v} \al_{s,t}$.  In particular, we have that \[\E[N_*(H)]\ge \sum_{(\b{x},\b{k},\b{y},\b{l})\in Q_H(4v)} \Pr[A(\b{x},\b{k},\b{y},\b{l})]\ge (8v)^{-2v}n^{2v}\cdot \al_v n^{-3v/2}=\Om(n^{v/2}).\]
		
		Now assume that $i=0$ and let $c=4m+4v$.  We will say that two vectors $\b{k},\b{l}$ each of length $m$ are nice if we have $4\le k_j\le \ell_j\le \floor{n/c}$ for all $j$.  Let $Q_c(\b{k},\b{l})$ denote the set of all quadruples $(\b{x},\b{k},\b{y},\b{l})$ such that
		\begin{align*}
		1+(2j-2+2v)\floor{n/c} \le\ &x_j \le-1+ (2j-1+2v)\floor{n/c},\\
		x_j+2 \le\ &y_j \le x_j+k_j-2,
		\end{align*}
		We claim that each quadruple of $Q_c(\b{k},\b{l})$ is good whenever $\b{k},\b{l}$ is nice.  The first condition follows since the largest point we pick is $y_m+\ell_m\le -1+(2m+2v)\floor{n/c}\le \f{n}{2}$ since $c=4m+4v$.  Similarly one can verify that \[x_j\le y_j-2\le x_j+k_j-4\le y_j+\ell_j-6\le x_{j+1}-8,\]
		where the first two inequalities follow from $x_j+2\le y_j\le x_j+k_j-2$, the third inequality from $\ell_j\ge k_j$ and $y_j\ge x_j+2$, and the last inequality from $y_j+\ell_j\le -1+(2j+2v)\floor{n/c}\le x_{j+1}-2$.  This shows that the second and third conditions of being a good quadruple are satisfied, proving the claim.   We also note that, for $n$ sufficiently large, \[|Q_c(\b{k},\b{l})|= (\floor{n/c}-1)^{m}\prod_{j=1}^m(k_j-3)\ge (8c)^{-m}n^m\prod_{j=1}^m k_j,\] where we've used that $k_j-3\ge \quart k_j$ for all $j$.
		
		Now let $H'$ denote $H$ after deleting its $m$ isolated edges.  For $\b{k},\b{l}$ nice, let $Q(\b{k},\b{l})$ be the set of all quadruples $q$ which are obtained by taking the union of the arcs of some $q_1\in Q_{H'}(c)$ and some $q_2\in Q_c(\b{k},\b{l})$.  We claim that every such $q$ is good.  Indeed, the first condition holds since it holds for both $q_1$ and $q_2$.  The second condition holds since it holds restricted to any two points of $q_1$ or $q_2$, and because the largest point of $q_1$ is at most $-1+2v\floor{n/c}$ while the smallest point of $q_2$ is at least $1+2v\floor{n/c}$.  This also implies that the third condition is satisfied since it is satisfied for both $q_1$ and $q_2$, so the claim is proven.
		
		Observe that each quadruple $(\b{x},\b{k},\b{y},\b{l})\in Q(\b{k},\b{l})$ represents $H$ and that \[\Pr[A(\b{x},\b{k},\b{y},\b{l})]\ge \al_{v} n^{-3(v-2m)/2} \prod_{j=1}^m k_j^{-3/2}\ell_j^{-3/2}\] by Lemma~\ref{L-GenLowProb}.  Also observe that our previous work shows that \[|Q(\b{k},\b{l})|=|Q_{H'}(c)|\cdot |Q_c(\b{k},\b{l})|\ge \be_{c} n^{2v-3e}\prod_{j=1}^m k_j\] for some absolute constant $\be_{c}$.  We conclude that
		\[
		\E[N_*(H)]\ge \sum_{\b{k},\b{l}\tr{ nice}}\sum_{(\b{x},\b{k},\b{y},\b{l})\in Q(\b{k},\b{l})} \Pr[A(\b{x},\b{k},\b{y},\b{l})]\ge \sum_{\b{k},\b{l}\tr{ nice}} \al_{v}\be_{c} n^{v/2} \prod_{j=1}^m k_j^{-1/2}\ell_j^{-3/2}\]\[=\al_{v}\be_{c} n^{v/2}\l(\sum_{4\le k\le \ell\le \floor{n/c}}k^{-1/2}\ell^{-3/2}\r)^m=\Om(n^{v/2}(\log n)^m),
		\]
		where we've used the fact that the above sum is of order $\Om(\log n)$.
		
		Now let $H$ be an arbitrary Catalan-pair graph.  Let $H''$ denote $H$ with its isolated vertices removed, and let $N'_*(H)$ denote the number of induced copies of $H''$ in $CP_n$ which have all of its points in the interval $[1,n/2]$.  Note that implicitly our above argument shows that $\E[N'_*(H)]=\Om(n^{(v-i)/2}(\log n)^m)$.
		
		We claim that, deterministically, $N_*(H)\ge N'_*(H)\cdot {n/4\choose i}$.  Indeed, observe that there are at most $n/2$ arcs which have an endpoint in the interval $[1,n/2]$, and hence there exists at least $n/2$ arcs with both endpoints not in this interval.  Let $A_R$ denote the set of these arcs that are colored red, and similarly define $A_B$.  One of these sets must have size at least $n/4$, so let $C$ be such that $|A_C|\ge n/4$.
		
		We claim that any induced copy of $H''$ contained in $[1,n/2]$ together with $i$ arcs of $A_C$ is an induced copy of $H$.  Indeed, by definition no arc in $A_C$ can interlace with any arc of the $H''$, and none of the $A_C$ arcs interlace with one another since they are all colored the same way.  Thus the graph that these arcs induce will be $H''$ together with $i$ isolated vertices, which is precisely $H$.  We conclude that
		\[
		N_*(H)\ge {|A_C|\choose i}\cdot N'_*(H'')\ge {n/4\choose i} N'_*(H'').
		\]
		The result now follows by taking expectations of the above inequality and using that $\E[N'_*(H)]=\Om(n^{(v-i)/2}(\log n)^m)$.
	\end{proof}

\subsection{An upper bound for the number of induced subgraphs} \label{subsec-upperbound}
A key step in finding the expected number of edges was to bound the number of good quadruples $(x,k,y,\ell)$ for given $k$ and $\ell$. Therefore, for general $H$ we would like to bound the number of valid quadruples $(\bf{x},\bf{y},\bf{k},\bf{l})$ for given $\bf{k}$ and $\bf{l}$. One of the reasons this is more complicated in the general setting is that $H$ might have several different representatives. However, since there are only finitely many representatives, it suffices to prove the desired bounds for each of them separately.

In order to do this we introduce some new notation. Let $H$ be a Catalan-pair graph on $v$ vertices and let $q = (\bar{\b{x}},\bar{\b{k}},\bar{\b{y}},\bar{\b{l}})$ be a quadruple with $\bar{\b{x}}$ and $\bar{\b{y}}$ increasing such that the following conditions are satisfied.
\begin{itemize}
	\item The lengths of $\bar{\b{x}}$ and $\bar{\b{y}}$ add to $v$.
	\item We have $\{\bar{x_i}\} \cup \{\bar{x_i}+\bar{k_i}\} \cup \{\bar{y_j}\} \cup \{\bar{y_j}+\bar{\ell_j}\} = \{1,2,\ldots,2v\}$.
	\item The quadruple $q$ is valid and the resulting Catalan-pair graph is isomorphic to $H$.
\end{itemize}
We say that a valid quadruple $(\b{x},\b{k},\b{y},\b{l})$ \emph{represents $H$ by $q$} if the relative order of the $x_i$, $x_i+k_i$, $y_j$ and $y_j+\ell_j$ coincides with the relative order of $\bar{x_i}$, $\bar{x_i}+\bar{k_i}$, $\bar{y_j}$ and $\bar{y_j}+\bar{\ell_j}$. Note that the $f_i$ and $g_j$ as defined in the beginning of Section~\ref{sec-edges} depend solely on $\b{k}$, $\b{l}$, and $q$, and are independent of the exact values of $\b{x}$ and $\b{y}$.

We wish to prove a lemma that upper bounds the number of valid quadruples for given $\b{k}$, $\b{l}$, and representing quadruple $q$. From now on we assume that $H$ is a connected Catalan-pair graph on $v \geq 3$ vertices that has $s$ and $t$ vertices in its bipartite components respectively. Additionally, let $q$ be a quadruple as above where $\bar{\b{x}}$ and $\bar{\b{y}}$ have length $s$ and $t$ respectively.

When $\b{k}$ and $\b{l}$ are known we denote by $(x_i)$ the arc $(x_i,x_i+k_i)$. For a valid quadruple $(\b{x},\b{k},\b{y},\b{l})$ we say that $(x_i)$ is a \emph{maximal arc} if there is no $j$ with $x_j < x_i < x_i+k_i < x_j+k_j$.  We say that arc $(x_i)$ covers arc $(x_j)$ if we have $x_i < x_j < x_j+k_j < x_i+k_i$ and there is no $i'$ with $x_i < x_{i'} < x_j < x_j+k_j < x_{i'}+k_{i'} < x_i+k_i$. Note that each arc is either maximal, or has a unique arc that covers it. However, a single arc can cover multiple arcs.

\begin{lem} \label{L-validquadruples-weak}
	Let $\b{k}$ and $\b{l}$ be $s$ and $t$-tuples of positive integers for which there exists a valid quadruple $(\b{x},\b{k},\b{y},\b{l})$ representing $H$ by $q$.  The number of such quadruples is at most
	\[
	(\min\{f_0,g_0\}+2v+1) \cdot \prod_{\substack{i \ge 1\\ i\ne i_0}} (f_i+2v+1) \cdot \prod_{j\ge 1} (g_j+2v+1)
	\]
	for any $i_0\ne 0$, and it also at most
	\[
		(f_0+2v+1)(g_0+2v+1) \cdot \prod_{\substack{i \ge 1\\ i\ne i_0}} (f_i+2v+1) \cdot \prod_{\substack{j \ge 1\\ j\ne j_0}} (g_j+2v+1)
	\]
	for any $i_0,j_0\ne 0$.
\end{lem}
\begin{proof}
	In order to prove the first bound we first consider the case that $f_0=\min\{f_0,g_0\}$. Let $i_0, i_1, \ldots, i_d$ be such that $(x_{i_d})$ is maximal and such that $(x_{i_p})$ covers $(x_{i_{p-1}})$ for all $1 \leq p \leq d$. We claim that there are at most
	\[
	(f_0 + 2v + 1) \cdot \prod_{p=1}^d (f_{i_p} + 2v + 1)
	\]
	ways to choose $x_{i_d}$, $x_{i_{d-1}}$, $\ldots$, $x_{i_1}$, $x_{i_0}$. Indeed, since we specified $q$, $\b{k}$, and $\b{l}$ (and hence the $f_i$ and $g_j$), we know how many points $m < x_{i_d}$ are of the form $m=y_j$, $m=y_{j}+\ell_j$, or which satisfy $x_{i'}\le m\le x_{i'}+k_{i'}$ for some $i'$.  By definition of $f_0$, we know that there are at most $f_0 \leq f_0 + 2v$ points outside of arc $x_{i_d}$ that are not of this form. We can choose amongst these at most $f_0+2v$ points how many lie to the left of $x_{i_d}$, and such a choice uniquely determines $x_{i_d}$ (since we now know the total number of points which lie to the left of $x_{i_d}$).  We conclude that we can place $x_{i_d}$ in at most $f_0+2v+1$ ways.  A similar argument shows that there are at most $f_{i_j}+2v+1$ ways to place each $x_{i_j-1}$ given that $x_{i_j}$ has already been placed, where now $f_{i_j}$ plays the role of $f_0$ by restricting our attention to points of the form $x_{i_j}<m<x_{i_j}+k_{i_j}$.  This completes the proof of the claim.
	
	Now suppose that we have inductively placed some (proper) subset of the arcs. Let $Z$ denote the set of arcs $z$ which have not been placed and whose endpoints alternate with some arc that has already been placed.  Since $H$ is connected, $Z\ne \emptyset$. Since $Z$ is finite, let $z\in Z$ be such that $z$ covers no other $z'\in Z$. Without loss of generality, assume that $z$ is of the form $(y_j)$. Then, we are in one of the following situations.
	\begin{itemize}
		\item[1.] The arc $(y_j)$ is minimal.
		\item[2.] The arc $(y_j)$ is not minimal and all the arcs covered by $(y_j)$ have been placed already.
		\item[3.] The arc $(y_j)$ is not minimal, at least one arc covered by $(y_j)$ has not been placed and any such arc does not alternate endpoints with any of the arcs placed so far.
	\end{itemize}
	We claim that in all cases there are at most $g_j + 2v + 1$ ways to choose $y_j$.
	\begin{itemize}
		\item[1.] Note that in this case there are at most $g_j + 2v$ points between $y_j$ and $y_j + \ell_j$. Indeed, there are $g_j$ points that are not of the form $x_i$ or $x_i+k_i$ and there are at most $2v$ points that are of this form. By assumption, the endpoints of $(y_j)$ alternate with the endpoints of some $(x_i)$. Consider the case where $y_j < x_i < y_j + \ell_j$.  Then the number of points between $y_j$ and $x_i$ is at most $g_j + 2v$, else there would be too many points between $y_j$ and $y_j + \ell_j$. Note that this number of intermediate points uniquely determines $y_j$ since $x_i$ is known. Therefore we have at most $g_j + 2v + 1$ ways to choose $y_j$.
		\item[2.] In this case we can follow a similar argument as used when choosing $x_{i_d}$. Note that since the $y_j$ are increasing, $y_{j+1}$ is the leftmost arc that is covered by $(y_j)$. By definition of $g_j$, there are at most $g_j + 2v$ points between $y_j$ and $y_{j+1}$ and the value of $y_j$ is known, so we again have at most $g_j + 2v + 1$ ways to choose $y_j$.
		\item[3.] In this case, suppose that $(y_j)$ intersects $(x_i)$ and that we have $y_j < x_i < y_j + \ell_j$. We again count the possible number of points between $y_j$ and $x_i$. As before, there are between $0$ and $g_j + 2v$ such points that do not lie below an arc covered by $(y_j)$. We claim that we know how many of the other points lie between $y_j$ and $x_i$, which again yields that there are at most $g_j + 2v + 1$ options for $y_j$.
		
		Indeed, consider an arc $(y_{j'})$ that is covered by $(y_j)$. If $(y_{j'})$ has not been placed, then it does not alternate endpoints with $(x_i)$ by assumption. Thus this arc either lies completely between $y_j$ and $x_i$ or completely between $x_i$ and $y_j + \ell_j$, and since we specified the quadruple $q$ representing $H$, we know which of these two cases happens.
		Thus we know exactly how many such points lie between $y_j$ and $x_i$. Now if $(y_{j'})$ has been placed, we know all of $y_{j'}$, $y_{j'} + \ell_{j'}$ and $x_i$, so clearly we also know how many of the points between $y_{j'}$ and $y_{j'} + \ell_{j'}$ lie to the left of $x_i$.
	\end{itemize}
	Inductively, we can place the arcs one by one (in the order described above) and note that in this process we get the product of all of the numbers of the form $f_i + 2v + 1$ and $g_j + 2v + 1$ except for the numbers $f_{i_0} + 2v + 1$ and $g_0 + 2m + 1$, establishing the first bound when $f_0=\min\{f_0,g_0\}$.
	
	Now assume that $g_0=\min\{f_0,g_0\}$.  Since $H$ is connected, there exists some $j_0\ne 0$ such that $(y_{j_0})$ and $(x_{i_0})$  interlace, and moreover we can choose $j_0$ such that it does not cover any $(y_{j'})$ that also interlaces with $(x_{i_0})$.  Let $j_0, j_1, \ldots, j_e$ be such that $(y_{j_e})$ is maximal and such that $(y_{j_p})$ covers $(x_{j_{p-1}})$ for all $1 \leq p \leq e$.  By the same reasoning as above, there are at most $(g_0 + 2v + 1) \cdot \prod_{p=1}^e (g_{i_p} + 2v + 1)$
	ways to choose $y_{j_e}$, $y_{j_{e-1}}$, $\ldots$, $y_{j_1}$, $y_{j_0}$.  We now place the remaining arcs $Z$ as we did before.  We use almost all of the same bounds as before, except we now use the bound $g_{j_0}+2v+1$ instead of $f_{i_0}+2v+1$ when we place $(x_{i_0})$.  We are justified in using this bound since, by assumption of $(y_{j_0})$ not covering any arc that interlaces with $(x_{i_0})$, one of the endpoints of $(x_{i_0})$ must be one of the points counted by $g_{j_0}$.  Ultimately this gives us the product of all of the numbers of the form $f_i + 2v + 1$ and $g_j + 2v + 1$ except for the numbers $f_{i_0} + 2v + 1$ and $f_0 + 2m + 1$ as desired.
	
	To prove the final bound, let $i_0, i_1, \ldots, i_d$ be such that $(x_{i_d})$ is maximal and such that $(x_{i_p})$ covers $(x_{i_{p-1}})$ for all $1 \leq p \leq d$, and similarly define $j_0,j_1,\ldots,j_e$.  By reasoning similar to that above, the number of ways we can place all of these arcs down in at most
	\[
	(f_0 + 2v + 1)(g_0+2v+1) \cdot \prod_{p=1}^d (f_{i_p} + 2v + 1)\cdot \prod_{p=1}^e (g_{i_p} + 2v + 1).
	\]
	We then place the remaining arcs and use the same bounds as we did before, and this ultimately gives us a product of all of the terms except for $f_{i_0}+2v+1$ and $g_{j_0}+2v+1$.
\end{proof}

\begin{prop} \label{P-validquadruples}
	Let $\b{k}$ and $\b{l}$ be $s$ and $t$-tuples of positive integers for which there exists a valid quadruple $(\b{x},\b{k},\b{y},\b{l})$ representing $H$ by $q$. Then the number of such quadruples is at most
	\[
	(h_1+2v+1) \cdot (h_2+2v+1)\cdot (h_4+2v+1) \cdot \prod_{i=6}^{v+2} (h_i+2v+1),
	\]
	where $h_1 \leq h_2 \leq \ldots \leq h_{v+1} \leq h_{v+2}$ are $f_0, f_1, \ldots, f_s$ and $g_0, g_1, \ldots, g_t$ written in increasing order.
\end{prop}
\begin{proof}
	Without loss of generality we may assume that $f_{i_0}=\max_{i,j\ne 0}\{f_i,g_j\}$.  Observe that $f_{i_0}\ge h_3$ since we assume $v\ge 3$, and further that $f_{i_0}\ge h_5$ if $\max\{f_0,g_0\}\le h_4$.  First assume that $\{f_0,g_0\}\ne \{h_1,h_2\}$.  In this case we apply the first bound of Lemma~\ref{L-validquadruples-weak} with our choice of $i_0$.  This bound consists of the product of all the values $h_i+2v+1$ except for the terms $f_{i_0}+2v+1$ and $\max\{f_0,g_0\}+2v+1$, and in this case we say that our bound ``omits'' the values $f_{i_0}+2v+1$ and $\max\{f_0,g_0\}+2v+1$.  If $\max\{f_0,g_0\}\ge h_5$ then these two terms are at least $h_3+2v+1$ and $h_5+2v+1$.  If $\max\{f_0,g_0\}\le h_4$, then we again omit at least $h_3+2v+1$ and $h_5+2v+1$ since $\{f_0,g_0\}\ne\{h_1,h_2\}$ implies that $\max\{f_0,g_0\} \ge h_3$.  Thus in this case we achieve our desired result.
	
	Now assume that $\{f_0,g_0\}=\{h_1,h_2\}$.  In this case we apply the second bound of Lemma~\ref{L-validquadruples-weak} to $i_0$ and $j_0=1$.  Now we omit only $f_{i_0}+2v+1$ (which is at least $h_5+2v+1$) and $g_1+2v+1$ (which is at least $h_3+2v+1$).  We conclude the result.
\end{proof}

With this proposition we can prove an upper bound on the expected number of induced subgraphs.

\begin{prop}\label{P-inducedUpper}
    Let $H$ be a connected Catalan-pair graph on $v\ge3$ vertices.  Then
    \[
        \E[N_H^*(CP_n)]=O(n^{v/2}).
    \]
\end{prop}
\begin{proof}
	First notice that there are only finitely many valid quadruples $q = (\bar{\b{x}},\bar{\b{k}},\bar{\b{y}},\bar{\b{l}})$ for which $\{\bar{x_i}\} \cup \{\bar{x_i}+\bar{k_i}\} \cup \{\bar{y_j}\} \cup \{\bar{y_j}+\bar{\ell_j}\} = \{1,2,\ldots,2v\}$ and such that the resulting Catalan-pair graph is isomorphic to $H$. Therefore, it suffices to show for each such $q$ that the expected number of induced Catalan-pair graphs of $CP_n$ that is represented by $q$ is $O(n^{v/2})$.
	
	Consider $1 \leq h_1 \leq h_2 \leq \ldots \leq h_{v+1} \leq h_{v+2} \leq 2n$. We claim that the number of pairs $(\b{k},\b{l})$ such that there exist a valid quadruple $(\b{x},\b{k},\b{y},\b{l})$ representing $H$ by $q$ and for which $\{h_i\} = \{f_i\} \cup \{g_j\}$ is at most $(v+2)!$. Indeed, note that since $q$ defines the relative order of all the points, knowing the values of $f_i$ and $g_j$ uniquely determines $\b{k}$ and $\b{l}$. Since there are $(v+2)!$ ways to distribute the $h_i$ over the $f_i$ and $g_j$, there are at most $(v+2)!$ possible pairs $(\b{k},\b{l})$.
	
	Therefore, using Lemma~\ref{L-upperbound} and Proposition~\ref{P-validquadruples} we find that the expected number of induced subgraphs isomorphic to $H$ and represented by $q$ is at most
	\begin{equation}\label{E-GenBound}
	(v+2)! \cdot \sum_h \left( (h_1+2v+1) \cdot (h_2+2v+1) \cdot (h_4+2v+1)\cdot\prod_{i=6}^{v+2} (h_i+2v+1) \cdot \beta_{s,t} n^3 \cdot \widetilde{\prod_i} h_i^{-3/2} \right)
	\end{equation}
	where the sum is over all possible sequences $h = (h_1, h_2, \ldots, h_{v+1}, h_{v+2})$ and $\widetilde{\prod}$ indicates the product over all $i$ with $h_i \geq 16 v \log n$. Note that implicitly this sum is over all possible $(\b{k},\b{l})$, and we will break up this sum into the cases where $\{\max f_i, \max g_j\} = \{h_a,h_{v+2}\}$ for all possible $a$. We will show the desired upper bound of $O(n^{v/2})$ in each of these cases. Note that $\sum f_i = 2n-2v$, so $\max f_i$ is at least linear and is uniquely determined by the other $f_i$.  We first consider $v\ge 5$.
	
	First, assume that $a \geq 6$. In this case, we can take out the factors $(h_a+2v+1) \cdot (h_{v+2}+2v+1) \cdot h_a^{-3/2} \cdot h_{v+2}^{-3/2}$ and note that this is $O(n^{-1})$, by virtue of $h_a, h_{v+2}$ being linear in $n$. Therefore, the remaining part can (up to some large constant) be estimated by
	\begin{equation}\label{E-Bound}
	n^2 \cdot \sum_{h_{v+1} = 1}^{2n} \cdots \sum_{h_{a+1}=1}^{h_{a+2}} \sum_{h_{a-1}=1}^{h_{a+1}} \cdots \sum_{h_1=1}^{h_2} (h_1+2v+1) \cdot (h_2+2v+1) \cdot(h_4+2v+1)\cdot \prod_{\substack{i=6 \\ i \neq a}}^{v+1} (h_i+2v+1)\cdot \widetilde{\prod_i} h_i^{-3/2},
	\end{equation}
	where the last product no longer involves $h_a$ nor $h_{v+2}$. Note that this expression is actually independent of $a$, so for simplicity we assume that $a = v+1$. Let $b$ be the number of $h_i$ for which $h_i \leq 16 v \log n$. First consider the case where $b = 0$. In this case, \eqref{E-Bound} is of the order
	\[
	n^2 \cdot \sum_{h_v = 1}^{2n} h_v^{-1/2} \sum_{h_{v-1}=1}^{h_v} h_{v-1}^{-1/2}  \cdots \sum_{h_6=1}^{h_7} h_6^{-3/2}\sum_{h_5=1}^{h_6} h_5^{-3/2} \sum_{h_4=1}^{h_5} h_4^{-1/2} \sum_{h_3=1}^{h_4} h_3^{-3/2} \sum_{h_2=1}^{h_3} h_2^{-1/2} \sum_{h_1=1}^{h_2} h_1^{-1/2}.
	\]
	Once again estimating these sums by integrals we find that
	\begin{align*}
	 \sum_{h_5=1}^{h_6} h_5^{-3/2} \sum_{h_4=1}^{h_5} h_4^{-1/2} \sum_{h_3=1}^{h_4} h_3^{-3/2} \sum_{h_2=1}^{h_3} h_2^{-1/2} \sum_{h_1=1}^{h_2} h_1^{-1/2} &= O\left( \sum_{h_5=1}^{h_6} h_5^{-3/2} \sum_{h_4=1}^{h_5} h_4^{-1/2}\sum_{h_3=1}^{h_4} h_3^{-3/2} \sum_{h_2=1}^{h_3} 1 \right) \\
	&= O\left( \sum_{h_5=1}^{h_6} h_5^{-3/2} \sum_{h_4=1}^{h_5} h_4^{-1/2} \sum_{h_3=1}^{h_4} h_3^{-1/2}\right) \\
	&= O\left( \sum_{h_5=1}^{h_6} h_5^{-3/2}\sum_{h_4=1}^{h_5} 1\right) \\
	&= O\left( \sum_{h_5=1}^{h_6} h_5^{-1/2}\right)=O(h_6^{1/2}) = O(n^{1/2}).
	\end{align*}
	Furthermore, each of the remaining sums is at most $\sum_{x=1}^{2n} x^{-1/2} = O(n^{1/2})$, so the total sum is $O(n^2 \cdot (n^{1/2})^{v-5} \cdot n^{1/2}) = O(n^{v/2})$.
	
	All of the cases $b=0$ and $2\le a\le 5$ have essentially the same proof as one another, so we will only explicitly go through one of these cases, namely $a=3$.  In this case we take out the factors $(h_{v+2}+2v+1)h_3^{-3/2}h_{v+2}^{-3/2}=O(n^{-2})$ from \eqref{E-Bound}, and we use the fact that $h_i\ge h_3$ is linear for all $i\ge 3$ to conclude \eqref{E-Bound} is of the order of magnitude at most
	\begin{align*}
	n \cdot \sum_{h_{v+1} = 1}^{2n} n^{-1/2} \cdots \sum_{h_6=1}^{2n} n^{-1/2} \sum_{h_5=1}^{2n} n^{-3/2} \sum_{h_4=1}^{2n} n^{-1/2} \sum_{h_2=1}^{2n} h_2^{-1/2} \sum_{h_1=1}^{h_2} h_1^{-1/2}&=O\l(n^{v/2-1}\sum_{h_2=1}^{2n} h_2^{-1/2} \sum_{h_1=1}^{h_2} h_1^{-1/2}\r)\\ &=O(n^{v/2}).
	\end{align*}
	
	Now consider the case that $a>b\geq 5$, and again we can assume for simplicity that $a=v+1$. Then \eqref{E-Bound} is at most of the order of
	\begin{align*}
	n^2 \cdot \sum_{h_v = 1}^{2n} & h_v^{-1/2} \sum_{h_{v-1}=1}^{h_v} h_{v-1}^{-1/2}  \cdots \sum_{h_{b+1}=1}^{h_{b+2}} h_{b+1}^{-1/2} \cdot \sum_{h_b=1}^{16 v \log n} (h_b+2v+1) \cdots \\
	& \cdots\sum_{h_6=1}^{16v\log n} (h_6+2v+1) \sum_{h_5=1}^{16 v \log n}  \sum_{h_4=1}^{16 v \log n}(h_4+2v+1) \sum_{h_3=1}^{16 v \log n} \sum_{h_2=1}^{16 m \log n} (h_2+2v+1) \sum_{h_1=1}^{16 v \log n} (h_1+2v+1).
	\end{align*}
	Note that each of the rightmost $b$ sums will contribute at most $O((\log n)^2)$ each, and the remaining sums will contribute $O(n^{(v-b)/2})$ by an argument similar to the one above. Thus the total contribution will be of the order $O( n^2 \cdot n^{(v-b)/2} \cdot (\log n)^{2b}) = o(n^{v/2})$.
	
	Similar arguments give a bound of $o(n^{v/2})$ when $b \in \{1,2,3,4\}$ and for any $a>b$. Note that since $h_a$ is linear in $n$, we always have $b < a$ for $n$ large enough, so these finitely many cases are all that need to be checked for $v=5$.  The proofs for $v=3,4$ are essentially the same, and we note that we did not deal with these cases earlier because we could not write, for example, $h_6$.  We omit the details.
\end{proof}
We note that the above proof shows the somewhat stronger result that the only quadruples that contribute to the order of magnitude of $n^{v/2}$ are those which have all of their gap sizes at least $16v \log n$.  With this we can now prove Theorem~\ref{T-induced}.

\begin{proof}[Proof of Theorem~\ref{T-induced}]
	The statement for induced subgraphs follows from Proposition~\ref{P-inducedLower} and \ref{P-inducedUpper}.  For any $H$ we claim that
	\[
	N_H^*(CP_n)\le N_H(CP_n) \leq v! \cdot \sum_{H'} N^*_{H'}(CP_n),
	\]
	where the sum is over all Catalan-pair graphs $H'$ on $v$ vertices that contain $H$ as a subgraph. The lower bound is obvious.  For the upper bound, note that for any given subgraph of $CP_n$ isomorphic to $H$, the induced subgraph on these vertices is isomorphic to some $H'$ appearing in this sum, and for given $H'$ there are at most $v!$ subgraphs of $H'$ isomorphic to $H$.  Taking the expectation of both sides of this inequality and using the result for induced subgraphs gives the desired conclusion.
\end{proof}
	
	.
\subsection{The sizes of the connected components}\label{SSec-components}

Computational evidence suggest that a typical random Catalan-pair graph on $n$ vertices will have one large component with roughly $n/2$ vertices, and a lot of smaller components. As we proved in Section~\ref{sec-isolated}, many of these components will be isolated vertices, but a significant amount will have larger size. In fact, we show that for \emph{any} fixed Catalan-pair graph the number of connected components of $CP_n$ isomorphic to this graph is linear in $n$.
\begin{prop} \label{T-components}
Let $H$ be a connected Catalan-pair graph on $v$ vertices and let $n \geq v+2$. There exists a constant $C$, independent of $H$, such that the expected number of connected components of $CP_n$ isomorphic to $H$ is at least $C \cdot (n-v+1/2) \cdot 16^{-v}$.
\end{prop}

\begin{proof}
Let $a$ and $A$ be as in \eqref{eq-aA} and take $C = \left(\frac{a}{A}\right)^2$. Assume that $H$ has bipartite components of sizes $s$ and $t$. We show that for any $1 \leq x \leq 2n-2v+1$, we have probability at least $1/2 \cdot (a/A)^2 \cdot 16^{-v}$ that there are $v$ arcs connecting $\{x,x+1,\ldots,x+2v-1\}$ and that the resulting Catalan-pair graph on these $2v$ points is isomorphic to $H$, which in particular yields a connected component of $CP_n$ isomorphic to $H$.

Consider a fixed representative for $H$. With probability $(1/2)^{2v}$ the points $x,x+1,\ldots,x+2v-1$ are colored in the exact same order as the points in the representative. Furthermore, since there are at least four other points, with probability at least $1/2$ the other points do not all have the same color. Therefore, we have $r > s$ and $b > t$ red and blue points in total. Given $r$ and $s$, the probability that we the arcs on the points $x, x+1, \ldots, x+2v-1$ exactly match those in the representative for $H$ is given by
\[
\frac12 \cdot \frac{C_{r-s}}{C_r} \cdot \frac{C_{b-t}}{C_b} \geq \frac12 \cdot \frac{a \cdot r^{3/2}}{4^s \cdot A \cdot (r-s)^{3/2}} \cdot \frac{a \cdot b^{3/2}}{4^t \cdot A \cdot (b-t)^{3/2}} \geq \frac12 \cdot \left(\frac{a}{A}\right)^2 \cdot \frac1{4^{s+t}}.
\]
Since $s+t = v$ this implies that with probability at least $4^{-v} \cdot \frac12 \cdot (a/A)^2 \cdot 4^{-v}$ we get such a connected component isomorphic to $H$ starting at point $x$. By linearity of expectation, the expected number of connected components isomorphic to $H$ is at least
\[
(2n-2v+1) \cdot 4^{-v} \cdot \frac12 \cdot \left(\frac{a}{A}\right)^2 \cdot 4^{-v} = (n-v+1/2) \cdot \left(\frac{a}{A}\right)^2 \cdot 16^{-v}. \qedhere
\]
\end{proof}

In particular, we expect a typical Catalan-pair graph on $n$ vertices to have connected components of size at least logarithmic in $n$.

\section{Computational Experiments} \label{sec-computational}
We consider some data from computer simulations of random Catalan-pair graphs. We do this both to provide visual evidence of some of the results we have proven, as well as to motivate further questions to be studied. In the first four graphs, each data point corresponds to averaging the given statistic over $100$ trials for $n = 100$, $200$, $\ldots$, $3000$ respectively.

The first graph shows the number edges of a random Catalan-pair graph divided by $n\log n$.  Since $\pi^{-1}\approx .318$, this data seems to suggest that the expected number of edges increases somewhat slowly to its asymptotic limit as proved in Theorem~\ref{T-edges}.

\includegraphics[scale=0.7]{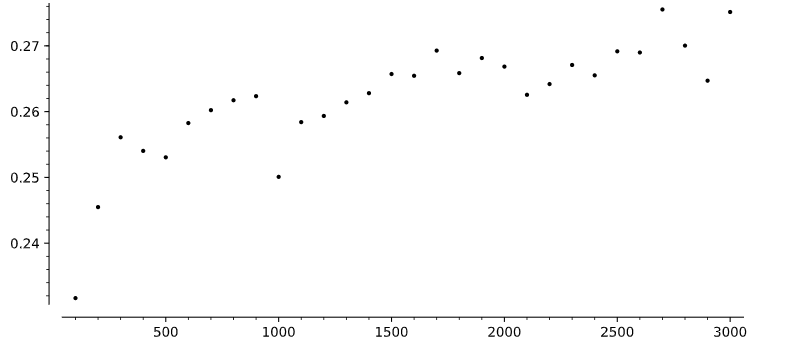}

The following graph shows the number of isolated vertices of a random Catalan-pair graph. The red plot corresponds to $0.3023 n$, in accordance with Theorem~\ref{T-isolated}.

\includegraphics[scale=0.7]{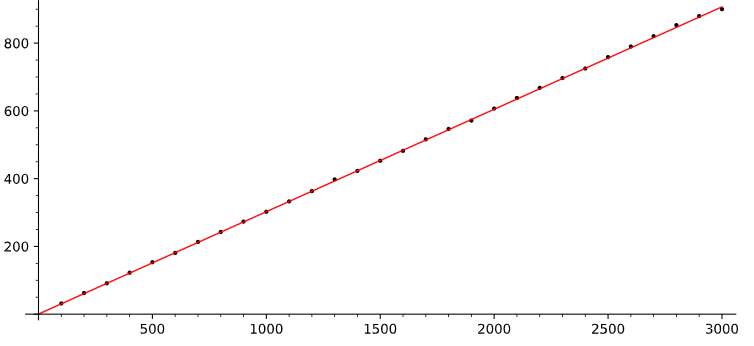}

The next two graphs show the sizes of the largest and second largest connected component respectively. The red plot corresponds to $0.55 n$.

\includegraphics[scale=0.7]{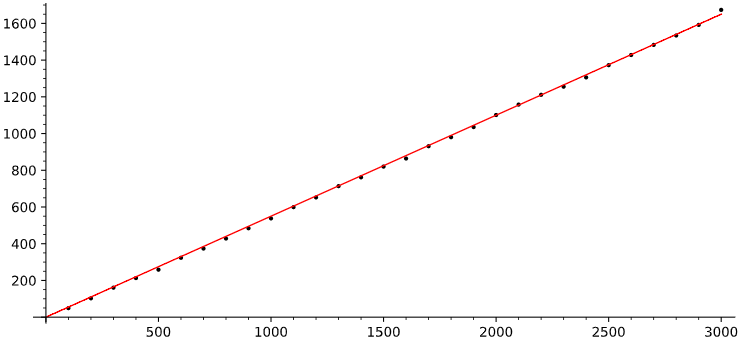}

The graph for the second largest component is still somewhat noisy, so we have not included a plot that tries to fit this data. Note that in Subsection~\ref{SSec-components} we suggest that the behavior should be at least logarithmic, but we likely require more data for larger $n$ to see if this is indeed the correct order of magnitude.

\includegraphics[scale=0.7]{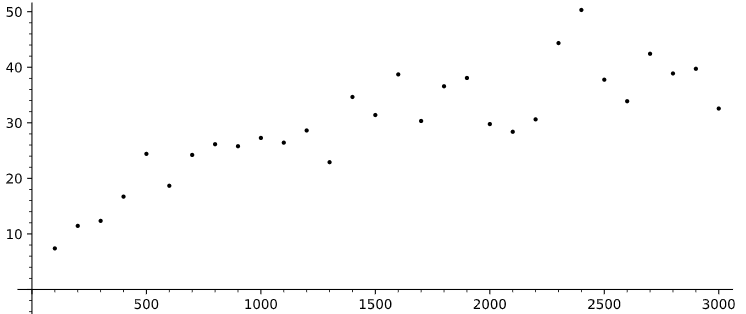}

We next look at four histograms of the distribution for the $100$ trials with $n = 3000$. We would like to point out that most of the histograms have their horizontal axis not starting at $0$.

First, we consider the total number of edges.  The binwidth for this plot is 60.

\includegraphics[scale=0.65]{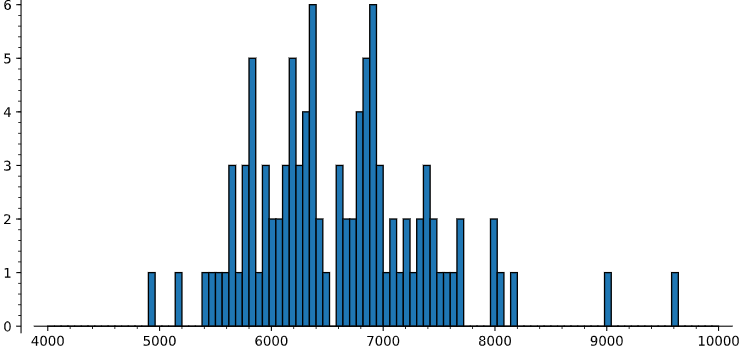}

Next we consider the sizes of the largest and second largest component, respectively.  We note that there are some outliers in the size of the second largest component in this data set, and this was also the case for several other data sets that we considered.  We have also observed noticeable outliers in the largest component in other data sets (on $n=1500$ vertices), though this could have been due to using too small a value of $n$.   The first plot has binwidth 19 and the second has binwidth 3.

\includegraphics[scale=0.65]{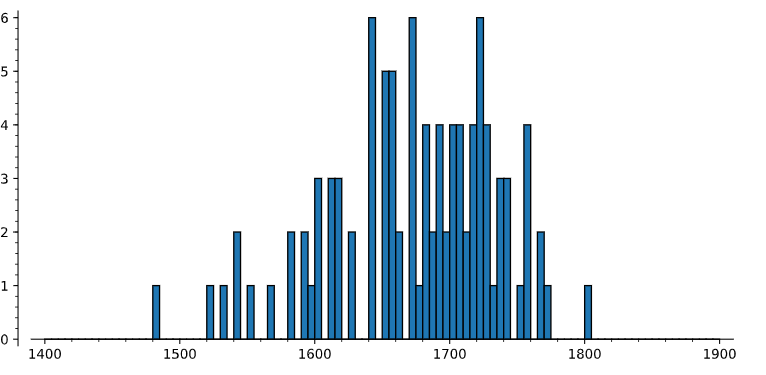}

\includegraphics[scale=0.65]{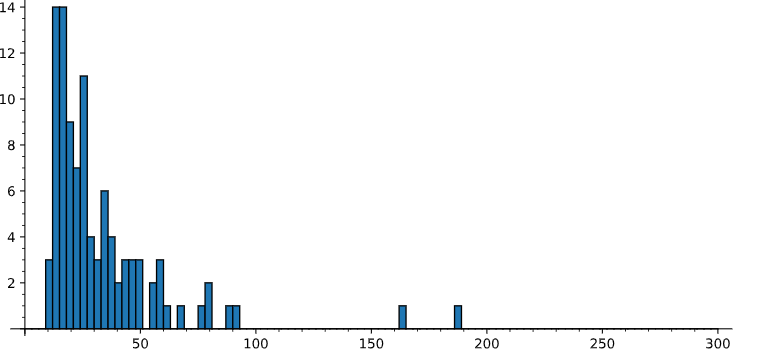}

We show the histogram for the total number of isolated vertices.  The binwidth for this plot is 2.

\includegraphics[scale=0.65]{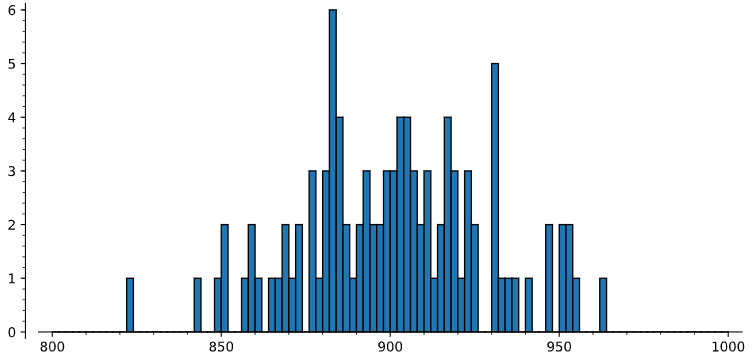}

We conclude this section with a look at the degree distribution of our 100 trials with $n=3000$.  The first plot shows the average number of vertices with a given degree that appeared during our 100 trials.

\includegraphics[scale=0.65]{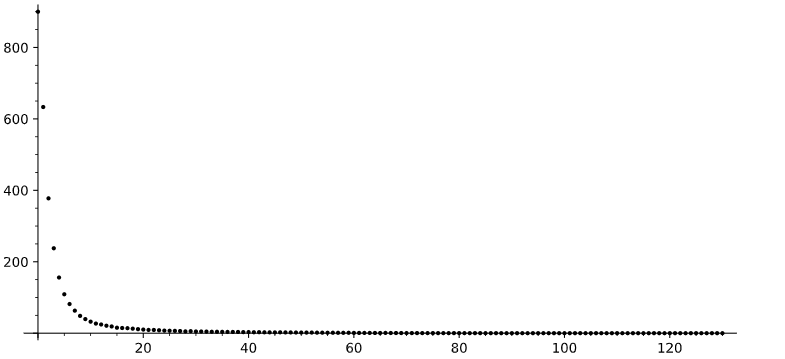}

Our final plot is a log-log plot of this data where our log is base $e$.

\includegraphics[scale=0.65]{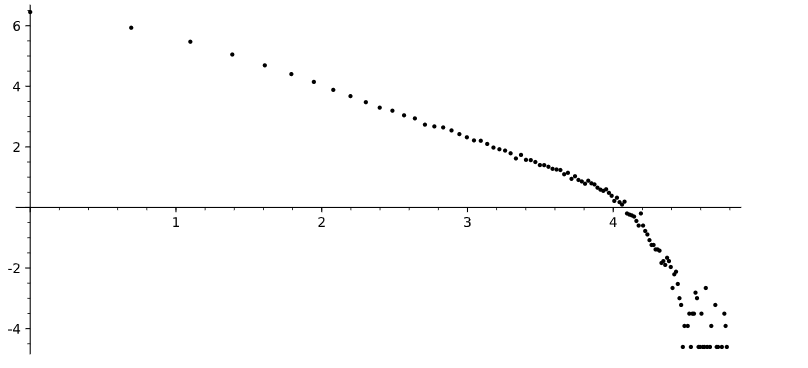}

We note that this plot appears mostly linear, which suggest that the degree distribution follows some power law distribution. However, the sharp turn at the end indicates that this behavior might only be valid for vertices of small enough degree.

\section{Conclusion and future problems} \label{sec-conclusion}
In this paper we introduced a model $CP_n$ for randomly generating Catalan-pair graphs, and we deduced various results concerning its subgraphs and connected components. There are many questions that remain to be explored. One such question is to investigate whether the lower bound in Proposition~\ref{P-inducedLower} holds for disconnected graphs as well.
\begin{prob}
    Determine the order of magnitude of $\E[N^*_H(CP_n)]$ when $H$ is a disconnected graph on at least $3$ vertices.
\end{prob}

In addition to the expectation, it would be of interest to determine (or at least bound) the second moments of random variables associated to $CP_n$. For example, it would be interesting to improve on Proposition~\ref{prop-isolated-variance} and Proposition~\ref{prop-edgevariance}.
\begin{prob}
    Determine more explicit bounds on the variance of the number of isolated vertices and the number of edges of $CP_n$.
\end{prob}
Such a result would be of interest as it would give an explicit bound on the concentration of these random variables around their mean by using the Chernoff bound. In order to improve on the concentration results in Theorem~\ref{T-edges} and Theorem~\ref{T-isolated} the following question would need to be answered as well.
\begin{prob}
    Determine explicit bounds on the quantities $|\E[e(CP_n) - \frac1{\pi} n \log n|$ and $|\E[I_n] - \gamma_n|$.
\end{prob}

While we have proven some results concerning the connected components of $CP_n$, there are many more questions that can be asked.  In particular we would like to know the following.
\begin{prob}
	What are the expected sizes of the largest components of $CP_n$?  Are the sizes of any of these components concentrated around their mean?
\end{prob}
Outliers in our computational evidence suggests that the second largest component might not have very strong concentration.  It is unclear whether or not this will be the case for the largest component.

The expected degree distribution of $CP_n$ remains unknown, though Theorem~\ref{T-components} does imply a lower bound for vertices of small degree.
\begin{prob}
	Describe the expected degree distribution of $CP_n$.  In particular, does it exhibit a power law distribution, possibly only for sufficiently small degrees?
\end{prob}

Lastly, we consider two additional models for randomly generating Catalan-Pair graphs which could be of interest.  These models are inspired by the random graph model $G(n,p)$, which is defined by including each possible edge of an $n$ vertex graph independently with probability $p$, as well as the model $\Gamma(n,m)$, defined by choosing uniformly at random a graph on $n$ vertices with exactly $m$ edges. For more details and result of these random models, see \cite{Gnp} for the model $G(n,p)$ and \cite{Gnm} for the model $\Gamma(n,m)$.

For $0\le p\le 1$, define $CP_n(p)$ the same way as we defined $CP_n$, but instead of coloring the first $2n-1$ colinear points red and blue with equal probability, we instead color each point red with probability $p$ and blue with probability $1-p$.  Essentially all our proofs carry over to $CP_n(p)$ when $p$ is a fixed constant, but it is not immediately clear how $CP_n(p)$ behaves when $p$ depends on $n$.

\begin{prob}
	What can be said about $CP_n(p)$ when $p$ depends on $n$?  Does $CP_n(p)$ exhibit evolutionary properties as $p$ grows?
\end{prob}

Another model to consider is $CP_n'(m)$, which is defined by coloring its $2n$ collinear points chosen uniformly from all colorings which have $2m$ red points, and then proceeding as in the definition of $CP_n$.  Intuitively, $CP_n'(m)$ and $CP_n(m/n)$ should behave in essentially the same way, at least when $m=\Theta(n)$.

In particular, we would like to be able to say that most results in $CP_n=CP_n(1/2)$ continue to hold in $CP'_n(n/2)$ and vice versa.  We believe that all of the proofs we have given in this paper can be modified without too much difficulty to work for $CP'_n(n/2)$ as well, though there will be technical difficulties.  For example, one should first prove that we have concentration results in $CP_n'(n/2)$ similar to those in $CP_n$.  A more subtle issue is that the probability that a given point is colored red or blue is not precisely $1/2$ in $CP'_n(n/2)$ once we have conditioned on other events occurring, so some care is needed to handle this, especially when dealing with asymptotic results.

Again, while we believe that on a case by case basis our results here carry over to $CP'_n(n/2)$, it would be nice if there was a more systematic way to accomplish this.  For example, we would like to say something analogous to the following statement relating $G(n,p)$ and $\Gamma(n,m)$, \cite[Thm. 7.6]{MGT}.
\begin{thm}
	Let $0<p=p(n)<1$ be such that $pn^2\to \infty$ and $(1-p)n^2\to \infty$, let $Q$ be a property of graphs, and let $\epsilon>0$ be fixed.

    If $(1-\epsilon){n\choose 2}<m<(1+\epsilon)p{N\choose 2}$ and asymptotically almost surely $\Gamma(n,m)$ has property $Q$, then asymptotically almost surely $G(n,p)$ has property $Q$.
\end{thm}

\begin{prob}
	Is there a systematic way to show that (reasonably nice) properties of $CP_n$ hold in $CP_n'(n/2)$ and vice versa?  More generally, can one show this for $CP_n(m/n)$ and $CP_n'(m)$ for various values of $m$?
\end{prob}

\section{Acknowledgments} \label{sec-acknowledgements}
The authors would like to thank Fan Chung Graham for suggesting this problem and for her guidance during the project.  We would also like to thank Joel Spencer for his helpful comments. The first author was partially supported by the Combinatorics Foundation.

\clearpage

\appendix
\section{Proofs of the edge variance}\label{appendix-edgevariance}
We will now provide the proofs of the lemmas in Section~\ref{sec-edgevariance}. First, we prove the lemma that concerns all pairs of edges coming from at most three arcs.

\begin{proof}[Proof of Lemma~\ref{L-variancefewarcs}.]
    Since it is clear that at least two arcs must be involved, there are two cases to consider. First, suppose that the total number of arcs involved equals two.  Then both edges in the pair are the same edge, so the number of such pairs equals $e(CP_n) \leq n^2$.
    On the other hand, if there are a total of three arcs involved, there are at most $n \cdot e(CP_n)$ pairs of such edges. Indeed, there are $e(CP_n)$ ways to choose the first edge in the pair, which yields two arcs, and then there are at most $n$ ways to choose a third arc that interlaces with either of the two arcs used already. Therefore, in expectation there are at most
    \[
    \E[n \cdot e(CP_n)] = \frac1{\pi} n^2 \log n = o(n^2 \log^2 n)
    \]
    such pairs.
\end{proof}

The next two lemmas are used to show that we may assume that each of the gap sizes is of order at least $\log n$. Before we give the proof let us define $e'(CP_n)$ as the number of edges in $CP_n$ at least one of whose arcs has size at most $d \log n$ or at least $c n$, which we will refer to as \emph{exceptional edges}. In Proposition~\ref{prop-edges} we showed that $\E[e'(CP_n)] = o(n \log n)$.

\begin{proof}[Proof of Lemma~\ref{L-smallnested}]
    First we consider the number of such pairs with $2n - k_1 < d \log n$. We claim that there are at most $(d \log n)^2 \cdot e(CP_n)$ such pairs. Indeed, we can pick the edge $(x_2,k_2,y_2,\ell_2)$ in at most $e(CP_n)$ ways, and the edge $(x_1,k_1,y_1,\ell_1)$ in at most $(d \log n)^2$ ways: we can pick $k_1$ in $d \log n$ ways, then since $k_1 > n$ there is at most one $x_1$ such that $x_1$ and $x_1 + k_1$ are connected and the vertex corresponding to this arc has degree at most $d \log n$ (as each interlacing arc must have an endpoint less than $x_1$ or larger than $x_1+k_1$). By taking expectations we see that we have at most $(d \log n)^2 \E[e(CP_n)] = o(n^2 \log^2 n)$ such pairs.

    Now suppose that $k_1 - k_2 < d \log n$ or $k_2 < d \log n$. First consider the pairs with $(x_1,k_1,y_1,\ell_1)$ an exceptional edge. We claim that the number of such pairs is at most $n \cdot d \log n \cdot e'(CP_n)$, from which taking expectations will suffice. In order to prove this, note that there are at most $e'(CP_n)$ ways to pick an exceptional edge. Then, in the case $k_1 - k_2 < d \log n$, there are at most $d \log n$ ways to pick $x_2$, and the corresponding arc has degree at most $n$. Similarly, if $k_2 <  d \log n$, there are at most $n$ ways to pick $x_2$, and the corresponding arc has degree at most $d \log n$.

    Therefore, we may assume that $(x_1,k_1,y_1,\ell_1)$ is not an exceptional edge. Assume that $k_1$ and $\ell_1$ are given. By the same logic as the proof of Proposition~\ref{prop-edgeasym}, we know that there are at most $4n \min\{k_1,\ell_1\}$ options for $x_1$ and $y_1$, and by Lemma~\ref{L-upperbound} the probability of having arcs connecting $(x_1,x_1+k_1)$ and $(y_1,y_1+\ell_1)$ is $O(k_1^{-3/2} \ell_1^{-3/2})$. Furthermore, given $(x_1,k_1,y_1,\ell_1)$ there are at most $k_1 \cdot d \log n$ possible second edges by a similar argument as above, where we now use $k_1$ instead of $n$ since we have fixed the size of the outer arc. Hence, the expected number of such pairs of edges is given by
    \[
    O\left( 2n \log n \cdot \sum_{k_1,\ell_1} \min\{k_1,\ell_1\} k_1^{-1/2} \ell_1^{-3/2} \right)
    \]
    so it suffices to show that $\sum_{k_1,\ell_1} \min\{k_1,\ell_1\} k_1^{-1/2} \ell_1^{-3/2} = o(n \log n)$. The contribution from $k_1 \leq \ell_1$ is at most
    \[
    \sum_{\ell_1 \leq cn} \ell_1^{-3/2} \sum_{k_1 \leq \ell_1} k_1^{1/2} \leq \sum_{\ell_1 \leq cn} \ell_1^{-3/2} O(\ell_1^{3/2}) = O(n),
    \]
    and the contribution from $\ell_1 \leq k_1$ is at most
    \[
    \sum_{k_1 \leq cn} k_1^{-1/2} \sum_{\ell_1 \leq k_1} \ell_1^{-1/2} = \sum_{k_1 \leq cn} k_1^{-1/2} O(k_1^{1/2}) = O(n)
    \]
    completing the proof.
\end{proof}

\begin{proof}[Proof of Lemma~\ref{L-smallunnested}]
    We first consider the case that one of $k_1$, $k_2$ is less than $d \log n$. By symmetry we can assume that $k_1 < d \log n$. As in the previous lemma, the number of pairs of edges with $(x_2,k_2,y_2,\ell_2)$ an exceptional edge is at most $n \cdot d\log n \cdot e'(CP_n)$ as there are at most $n \cdot d\log n$ edges where one vertex has degree at most $d \log n$ and there are at most $e'(CP_n)$ ways to pick the second edge. Therefore, in expectation, there are at most $O(n \log n) \cdot \E[e'(CP_n)] = o(n^2 \log^2 n)$ such pairs.

    Thus we may assume that $(x_2,k_2,y_2,\ell_2)$ is not an exceptional edge. Consider all pairs of edges where $k_1 < \sqrt{\log n}$. The number of such pairs is at most $n \cdot \sqrt{\log n} \cdot e(CP_n)$, as one can pick the arc $(x_1,x_1+k_1)$ in at most $n$ ways, this vertex has degree at most $\sqrt{\log n}$, and there are at most $e(CP_n)$ ways to pick the second edge. In particular, the expected number of such pairs is at most $n \cdot \sqrt{\log n} \cdot \E[e(CP_n)] = O(n^2 (\log n)^{3/2}) = o(n^2 \log^2 n)$.

    Lastly we handle the case where $\sqrt{\log n} \leq k_1 \leq d \log n$. We consider the expected number of pairs of an arc and an edge $((x_1,k_1), (x_2,k_2,y_2,\ell_2))$ such that $k_1$ is in the given range, and the arcs $(x_1,x_1+k_1)$ and $(x_2,x_2+k_2)$ are not nested. If we can show that the expected number of such pairs is $o(n^2 \log n)$ the result follows.  Indeed, any pair of edges of interest comes from such an arc-edge pair together with an arc that interlaces with $(x_1,k_1)$, and there are at most $O(\log n)$ such arcs.  Thus in total we will get at most $o(n^2 \log n) \cdot O(\log n) = o(n^2 \log^2 n)$ pairs of edges.

    To accomplish this, we will show that for any valid quadruple $q = ((x_1,x_2),(k_1,k_2),(y_2),(\ell_2))$ giving an arc-edge pair as described above, we have
    \begin{equation} \label{eq-unnestedsmall}
    \Pr[A(q)] = O\left(k_2^{-3/2} \ell_2^{-3/2} \cdot \left(k_1^{-3/2} + e^{-\sqrt{\log n}/16}\right)\right).
    \end{equation}
    Showing the above bound on the probability suffices because then the number of arc-edge pairs is at most
    \[
    \sum_q \Pr[A(q)] = O\left(\left(\sum_{x_1,k_1} k_1^{-3/2} + e^{-\sqrt{\log n}/16}\right)\cdot\left(\sum_{x_2,k_2,y_2,\ell_2} k_2^{-3/2} \ell_2^{-3/2} \right)\right)
    \]
    where we note that some combinations of some $(x_1,k_1)$ used in the first sum and some $(x_2,k_2,y_2,\ell_2)$ used in the second sum will not give a desired quadruple $q$, but this is no issue since we are only interested in an upper bound. By Lemma~\ref{L-Range} and Proposition~\ref{prop-edgeasym} the first sum is $o(n)$ and the second sum is $O(n \log n)$, showing the desired result.  We will deviate slightly and assume that $\ell_2$ is at least $2d \log n$, but we note that this change will not affect our previous arguments.

    To prove \eqref{eq-unnestedsmall} we note that $\Pr[A(q)]$ can be written as
    \begin{equation} \label{eq-unnestedsmall2}
    \Pr[A(q)] = 2^{-2n} \sum_c \frac{C_{n_0} C_{n_1} C_{n_2}}{C_{n_0+n_1+n_2+2}} \cdot \frac{C_{m_0} C_{m_2}}{C_{m_0+m_2+1}},
    \end{equation}
    where the sum is over all colorings $c$ of the points such that all the points coming from $q$ receive the correct color and the number of points of the desired color in each region is even. Here $n_0$ and $m_0$ are half the number of red and blue points outside of the desired arcs, $n_1$ is half the number of red points within arc $(x_1,x_1+k_1)$ and $n_2$ and $m_2$ are half the number of red an blue points respectively in the arcs $(x_2,x_2+k_2)$ and $(y_2,y_2+\ell_2)$. Note that $\ell_2 > 2d \log n$, hence the number of points between $y_2$ and $y_2+\ell_2$ that do not lie between $x_1$ and $x_1 + k_1$ is at least $d \log n$.

    Consider all the possible colorings of all the points except for the points in the interval $[x_1,x_1+k_1]$. By using Lemma~\ref{L-Chernoff}, for $d$ large enough, we can say that with probability at least $1 - O(n^{-10})$ we have $n_0, m_0 = \Omega(n)$, $n_2 = \Omega(k_2)$ and $m_2 = \Omega(\ell_2)$, where the bound on $m_2$ follows by the above remark that there are still at least $d \log n$ points that we are considering. Since
    \[
    n^{-10} = o\left(\left(k_1^{-3/2} + e^{-\sqrt{\log n}/16}\right)\cdot\left(k_2^{-3/2} \ell_2^{-3/2} \right)\right)
    \]
    we can restrict our attention to all colorings where the above bounds are satisfied. Now, for any such coloring, using the asymptotic formula for the Catalan numbers, we have
    \[
    \frac{C_{m_0} C_{m_2}}{C_{m_0+m_2+1}} = O(\ell_2^{-3/2}).
    \]
    Furthermore, we can rewrite
    \[
    \frac{C_{n_0} C_{n_1} C_{n_2}}{C_{n_0+n_1+n_2+2}} = \frac{C_{n_0+n_2+1} C_{n_1}}{C_{n_0+n_1+n_2+2}} \cdot \frac{C_{n_0} C_{n_2}}{C_{n_0+n_2+1}},
    \]
    then as in Lemma~\ref{L-Range} we can show that $\frac{C_{n_0+n_2+1} C_{n_1}}{C_{n_0+n_1+n_2+2}}$, which is the probability of having an arc connecting $x_1$ and $x_1+k_1$, is given by $O\left(k_1^{-3/2} + e^{-\sqrt{\log n}/16}\right)$, where this case is even a bit easier since we already specified the number of red points outside the arc. Furthermore, plugging in $n_0 = \Omega(n)$ and $n_2 = \Omega(k_2)$ we find $\frac{C_{n_0} C_{n_2}}{C_{n_0+n_2+1}} = O(k_2^{-3/2})$, and plugging all these results into \eqref{eq-unnestedsmall2} yields
    \[
    \Pr[A(q)] = 2^{-2n} \sum_c O\left(\left(k_1^{-3/2} + e^{-\sqrt{\log n}/16}\right)\cdot\left( k_2^{-3/2} \ell_2^{-3/2} \right)\right),
    \]
    which is $O\left(\left(k_1^{-3/2} + e^{-\sqrt{\log n}/16}\right)\cdot\left(k_2^{-3/2} \ell_2^{-3/2} \right)\right)$ since there are at most $2^{2n}$ valid colorings $c$.  The finishes the case that one of $k_1,k_2$ is less than $d\log n$.

    Secondly, consider the case that $2n - (k_1+k_2) < d \log n$. By the above we may assume that $k_1, k_2 > d \log n$. For any $d \log n < k_1 < 2n - d \log n$ there are at most $O(\log n)$ values of $k_2$ for which $2n - (k_1+k_2)$ is satisfied. Furthermore, given $k_1$ and $k_2$ there are at most $O((\log n)^2)$ ways to pick $x_1$ and $x_2$, as there are at most $d \log n$ dots outside of the arcs $(x_1,x_1+k_1)$ and $(x_2,x_2+k_2)$. A variant of the proof of Lemma~\ref{L-upperbound} shows that with probability $O(n^{3/2} k_1^{-3/2} k_2^{-3/2})$ we have arcs connecting $x_1$ and $x_1+k_1$, and $x_2$ and $x_2+k_2$.

    Given $k_1$, $k_2$, $x_1$ and $x_2$, and assuming that $(x_1,x_1+k_1)$ and $(x_2,x_2+k_2)$ match there are at most $k_1 \cdot k_2$ edges involving these two arcs. Therefore, the expected number of pairs of edges is at most
    \[
    \sum_{k_1,k_2} O((\log n)^2) \cdot O(n^{3/2} k_1^{-3/2} k_2^{-3/2}) \cdot k_1 k_2 = O(n^{3/2} (\log n)^2) \sum_{k_1,k_2} k_1^{-1/2} k_2^{-1/2}.
    \]
    We now claim that $k_2 \geq \frac12 (2n - k_1)$. Indeed, if $k_1 \geq 2n - 2d \log n$ we have $\frac12 (2n-k_1) \leq d \log n$, whereas $k_2 \geq d \log n$. Otherwise, we have $k_2 \geq 2n - k_1 - d \log n \geq \frac12 (2n - k_1)$ since the last inequality is equivalent to $k_1 \leq 2n - 2d \log n$. Using this, together with the earlier observation that there are at most $O(\log n)$ choices for $k_2$ given $k_1$, we find
    \begin{align*}
    O(n^{3/2} (\log n)^2) \sum_{k_1,k_2} k_1^{-1/2} k_2^{-1/2}
        &= O(n^{3/2} (\log n)^2) \sum_{k_1,k_2} k_1^{-1/2} (2n - k_1)^{-1/2} \\
        &= O(n^{3/2} (\log n)^3) \sum_{k_1} k_1^{-1/2} (2n-k_1)^{-1/2}.
    \end{align*}
    Using that $x \mapsto (x(2n-x))^{-1/2}$ is decreasing on $(0,n)$ and increasing on $(n,2n)$ we can compare the last sum with an integral to find that
    \begin{align*}
    \sum_{k_1} k_1^{-1/2} (2n-k_1)^{-1/2} &\leq \int_1^{2n-1} (x(2n-x))^{-1/2} \ \mathrm{d}x = \left. 2 \arctan\left(\sqrt{\frac{x}{2n-x}}\right) \right|_1^{2n-1} \\
    &= 2 \arctan(\sqrt{2n-1}) - 2 \arctan\left(\sqrt{\frac1{2n-1}}\right) \leq \pi.
    \end{align*}
    Therefore, the expected number of pairs of these edges is at most $O(n^{3/2} (\log n)^3) = o(n^2 \log^2 n)$.
\end{proof}

The next lemma takes care of the cases where the arcs on at least one side are nested.

\begin{proof}[Proof of Lemma~\ref{L-variancenested}]
    There are three cases to consider, based on the relative position of the arcs coming from the bottom:
    \begin{itemize}
        \item[1.] These arcs are unnested.
        \item[2.] We have $y_2 < y_1 < y_1+\ell_1 < y_2+\ell_2$.
        \item[3.] We have $y_1 < y_2 < y_2 + \ell_2 < y_1 + \ell_1$.
    \end{itemize}

    We will prove that in each case we have
    \begin{equation} \label{eq-nested}
    \Pr[A(\b{x},\b{k},\b{y},\b{l})] = O\left(n^3 k_1^{-3/2} (2n-k_1)^{-3/2} \ell_1^{-3/2} (2n-\ell_1)^{-3/2} k_m^{-3/2} \ell_m^{-3/2}\right),
    \end{equation}
    where $k_m = \min\{k_2-k_1,k_2\}$ and $\ell_m$ is defined based on which of the three cases we are working in. Furthermore, in all cases we will show an upper bound of $O(g(k_1,\ell_1) \cdot k_1 \cdot \min\{k_m,\ell_m\})$ on the number of choices for $x_1,x_2,y_1,y_2$ given $k_1,k_2,\ell_1,\ell_2$. Here $g(k,\ell)$ is the number of pairs $(x,y)$ such that $(x,k,y,\ell)$ is a good quadruple, as defined in Proposition~\ref{prop-edgeasym} . We note that given $k_m$ and $k_1$ there are only two possibilities for $k_2$ and we will define $\ell_m$ in such a way that the same thing holds for $\ell_2$ given $\ell_m$ and $\ell_1$. Therefore, the desired contribution will be of the order
    \[
    \sum_{k_1,\ell_1,k_m,\ell_m} g(k_1,\ell_1) \cdot k_1 \cdot \min\{k_m,\ell_m\} \cdot n^3 k_1^{-3/2} (2n-k_1)^{-3/2} \cdot \ell_1^{-3/2} (2n-\ell_1)^{-3/2} k_m^{-3/2} \ell_m^{-3/2}.
    \]
    Simply allowing all the variables in this sum to run between $d \log n$ and $2n - d \log n$ we can factor this as
    \[
    \left(\sum_{k_m,\ell_m} \min\{k_m,\ell_m\} k_m^{-3/2} \ell_m^{-3/2}\right)\cdot \left(\sum_{k_1,\ell_1} g(k_1,\ell-1) \cdot k_1 \cdot n^3 k_1^{-3/2} (2n-k_1)^{-3/2} \ell_1^{-3/2} (2n-\ell_1)^{-3/2}\right).
    \]
    Note that the first sum is of order $O(\log n)$. Now, if $\max\{k_1,\ell_1\} > cn$ we can use the estimate $k_1 = O(n)$, to show that the total contribution is given by
    \[
    O(n \log n) \cdot \left(\sum_{k_1,\ell_1} g(k_1,\ell-1)  \cdot n^3 k_1^{-3/2} (2n-k_1)^{-3/2} \ell_1^{-3/2} (2n-\ell_1)^{-3/2}\right) = o(n^2 \log^2 n),
    \]
    as the last sum is of order $o(n \log n)$ by Proposition~\ref{prop-edges}. Else, we can use $n^3(2n-k_1)^{-3/2}(2n-\ell_1)^{-3/2} = O(1)$ and the estimate $g(k_1,\ell_1) \leq 4n \min\{k_1,\ell_1\} \leq 4n \ell_1$ to see that the total contribution is of the order
    \[
    O(n \log n) \cdot \left(\sum_{k_1,\ell_1} k_1^{-1/2} \ell_1^{-1/2}\right) = O(n^2 \log n) = o(n^2 \log^2 n),
    \]
    where we used that
    \[
    \sum_{k_1,\ell_1} k_1^{-1/2} \ell_1^{-1/2} = \left(\sum_{d \log n \leq k_1 \leq cn} k_1^{-1/2}\right) \cdot \left(\sum_{d \log n \leq k_1 \leq cn} k_1^{-1/2}\right) = O(\sqrt{n}) \cdot O(\sqrt{n}).
    \]

    We now show \eqref{eq-nested} and the desired bounds on the number of quadruples for each of the cases. We handle the first case in full detail, the other two cases are very similar so we only highlight the details.
    \begin{itemize}
        \item[1.] We know from Lemma~\ref{L-upperbound} that
            \[
            \Pr[A(\b{x},\b{k},\b{y},\b{l})] = O\left(n^3 (2n-k_1)^{-3/2} (k_1-k_2)^{-3/2} k_2^{-3/2} (2n-\ell_1-\ell_2)^{-3/2} \ell_1^{-3/2} \ell_2^{-3/2}\right).
            \]
            In this case, we define $\ell_m = \min\{2n-\ell_1-\ell_2,\ell_2\}$. Now, since $(k_1-k_2)+k_2 = k_1$ we have $\max\{k_1-k_2,k_2\} \geq k_1/2$, so $(k_1-k_2)^{-3/2} k_2^{-3/2} = O(k_1^{-3/2} k_m^{-3/2})$ and similarly we find $(2n-\ell_1-\ell_2)^{-3/2} \ell_2^{-3/2} = O((2n-\ell_1)^{-3/2} \ell_m^{-3/2})$.

            Furthermore, given $k_1,k_2,\ell_1,\ell_2$ there are at most $g(k_1,\ell_1) + O(n) = O(g(k_1,\ell_1))$ ways to pick $(x_1,y_1)$, where we have to add $O(n)$ to account for the option that $(x_1,k_1,y_1,\ell_1)$ is not a good quadruple. Now suppose that $(x_1,y_1)$ has been chosen.

            If $k_m < \ell_m$ there are at most $(k_1-k_2)$ ways to pick $x_2$ and after that at most $2k_2$ ways to pick $y_2$, so there are at most $O((k_1-k_2) k_2) = O(k_1 k_m)$ ways to pick $(x_2,y_2)$ (where we used $k_1-k_2, k_2 \leq k_1$).

            Similarly, if $\ell_m < k_m$ there are at most $k_1$ ways to pick $x_2$ and we claim that there are at most $O(\ell_m)$ ways to pick $y_2$. Indeed, if $\ell_m = 2n - \ell_1 - \ell_2$ then there are at most two ways to pick the relative order of the arcs, after which $y_1$ is determined by how many of the $2n - \ell_1 - \ell_2 = \ell_m$ outside points are to the left of $y_2$, whereas if $\ell_m = \ell_2$ the value of $y_2$ is determined by the relative order of $x_2$ and $y_2$ and by how many points the arcs $(x_2,x_2+k_2)$ and $(y_2,y_2+\ell_2)$ have in common. For the first option we have two choices and for the last one we have $\ell_2 =\ell_m$ choices.
        \item[2.] In this case we have
            \[
            \Pr[A(\b{x},\b{k},\b{y},\b{l})] = O\left(n^3 (2n-k_1)^{-3/2} (k_1-k_2)^{-3/2} k_2^{-3/2} (2n-\ell_2)^{-3/2} \ell_1^{-3/2} (\ell_2-\ell_1)^{-3/2}\right),
            \]
            so defining $\ell_m = \min\{2n-\ell_2,\ell_2-\ell_1\}$ gives the desired bound on the probability.

            For the count of the number of options for $(x_1,y_1,x_2,y_2)$ the only thing that changes is the number of ways to pick $(x_2,y_2)$ given $(x_1,y_1)$ and given $\ell_m \leq k_m$. Again, there are at most $k_1$ ways to pick $x_2$. If $\ell_m = \ell_2-\ell_1$ then $y_2$ is determined by the number of dots between $y_1$ and $y_2$, whereas if $\ell_m = 2n - \ell_2$ the value of $y_2$ is determined by choosing how many of the outside points should be to the left of $y_2$.
        \item[3.] Here we have the bound
            \[
            \Pr[A(\b{x},\b{k},\b{y},\b{l})] = O\left(n^3 (2n-k_1)^{-3/2} (k_1-k_2)^{-3/2} k_2^{-3/2} (2n-\ell_1)^{-3/2} \ell_2^{-3/2} (\ell_1-\ell_2)^{-3/2}\right),
            \]
            so we define $\ell_m = \min\{\ell_1-\ell_2,\ell_2\}$.

            Again, the only thing that remains is to bound the number of ways to pick $y_2$ given $(x_1,y_1,x_2)$ in the case $\ell_m \leq k_m$. If $\ell_m = \ell_1-\ell_2$ then $y_2$ is determined by picking the distance between $y_1$ and $y_2$, whereas if $\ell_m = \ell_2$ the value of $y_2$ is determined by picking the relative order of $x_2$ and $y_2$ and choosing the number of points that the two arcs $(x_2,x_2+k_2)$ and $(y_2,y_2+\ell_2)$ have in common. \qedhere
    \end{itemize}
\end{proof}

Next we handle the case where at least one of the arcs has size linear in $n$.

\begin{proof}[Proof of Lemma~\ref{L-variancelong}]
    Without loss of generality we assume that $k_1 = \max\{k_1,k_2,\ell_1,\ell_2\}$. Let $k_0 = 2n-k_1-k_2$ and $\ell_0 = 2n-\ell_1-\ell_2$ and set $m_i = \min\{k_i,\ell_i\}$ for $i = 0,1,2$. First assume that $\ell_1 \neq \max\{ \ell_0, \ell_1, \ell_2\}$.

    We claim that given $k_1,k_2,\ell_1,\ell_2$, the number of quadruples is at most $O(m_0^2 m_1 m_2) = O(m_0^2 \ell_1 m_2)$. Since there are only finitely many options for the orderings of the endpoints of the arcs, it suffices to show the bounds for each specific ordering. But, given the ordering of the arcs, we claim that there are at most $m_0m_i$ ways to pick $(x_i,y_i)$. Indeed, consider the case that $k_0 = \min\{k_0,\ell_0\}$. Then we can pick $x_i$ in at most $k_0$ ways, as it is determined by the number of points to the left of $x_i$ (if the arc $(x_i,x_i+k_i)$ is the leftmost arc) or to the number of points to the right of $x_i+k_i$ (if the arc is the rightmost arc), so $x_i$ can be picked in at most $k_0 = m_0$ ways. After that, $y_i$ is determined by the number of points that the arcs $(x_i,x_i+k_i)$ and $(y_i,y_i+\ell_i)$ have in common and this is at most $m_i$. The case $\ell_0 = \min\{k_0,\ell_0\}$ is similar.

    Now given a quadruple, by Lemma~\ref{L-upperbound} the probability that that all the desired arcs match is $O(n^{3/2} k_0^{-3/2} \ell_0^{-3/2} \ell_1^{-3/2} k_2^{-3/2} \ell_2^{-3/2})$ where we used that $k_1 \geq cn$. Therefore, the desired contribution is at most
    \begin{equation} \label{eq-unnestedbig}
    O(n^{3/2}) \cdot \sum m_0^2 \ell_0^{-3/2} k_0^{-3/2} \cdot \ell_1^{-1/2} \cdot m_2 k_2^{-3/2} \ell_2^{-3/2}.
    \end{equation}
    Since $\ell_0 + \ell_1 + \ell_2 = 2n$ we have $\max\{\ell_0,\ell_1,\ell_2\} \geq 2n/3$. Since we assumed that $\ell_1$ is not the maximum we have two cases.

    \begin{itemize}
        \item $\ell_0$ is the maximum. In this case $n^{3/2} \ell_0^{-3/2} = O(1)$. Note that (as we did in Proposition~\ref{P-inducedUpper}) $\ell_0$ is determined by $\ell_1$ and $\ell_2$, and $k_1$ is determined by $k_0$ and $k_2$, so its contribution to \eqref{eq-unnestedbig} is
        \[
        O(1) \cdot \sum_{k_0,k_2,\ell_1,\ell_2} m_0^2 k_0^{-3/2} \ell_1^{-1/2} m_2 k_2^{-3/2} \ell_2^{-3/2},
        \]
        where the sum is over some appropriate range. To find an upper bound we can split this sum as
        \[
        O(1) \cdot \left(\sum_{k_0} m_0^2 k_0^{-3/2}\right) \cdot \left(\sum_{\ell_1} \ell_1^{-1/2} \right) \cdot \left(\sum_{k_2,\ell_2} m_2 k_2^{-3/2} \ell_2^{-3/2}\right),
        \]
        which after merging back involves more terms than before, but that is fine as we are only interested in an upper bound. We will now estimate each individual sum. For the first one, if $k_0 \leq \ell_0$ this contributes $\sum_{k_0} k_0^{1/2} = O(n^{3/2})$, whereas if $k_0 \geq \ell_0$ this sum is at most $O(n^2) \sum k_0^{-3/2} = O(n^2) \cdot O(n^{-1/2}) = O(n^{3/2})$ where we used that $k_0 \geq 2n/3$ in this case. For the second sum we get a bound of $O(n^{1/2})$. For the last sum we may assume $k_2 \leq \ell_2$ by symmetry and see that this sum is
        \[
        O\left(\sum_{\ell_2} \ell_2^{-3/2} \sum_{k_2 \leq \ell_2} k_2^{-1/2}\right) = O\left(\sum_{\ell_2} \ell_2^{-1}\right) = O(\log n),
        \]
        so all together we get $O(n^2 \log n)$ in this case.
        \item Now assume that $\ell_2 = \max\{\ell_0,\ell_1,\ell_2\}$. Using the estimate $O(n^{3/2}) \cdot \ell_2^{-3/2} = O(1)$ the contribution to \eqref{eq-unnestedbig} is at most
        \[
        O\left(\left(\sum_{k_0,\ell_0} m_0^2 k_0^{-3/2} \ell_0^{-3/2}\right) \cdot \left(\sum_{\ell_1} \ell_1^{-1/2}\right) \cdot \left(\sum_{k_2} m_2 k_2^{-3/2} \right)\right).
        \]
        Similar arguments to above give that the first sum is $O(n)$, the second one is $O(n^{1/2})$ and the last one is $O(n^{1/2})$ where here one has to distinguish cases based on whether $k_2 \geq \ell_2$ or $k_2 \leq \ell_2$ just as for the first sum in the case above, so the total contribution will be $O(n^2) = o(n^2 \log^2 n)$, as desired.
    \end{itemize}
It remains to handle the case $\ell_1 = \max\{\ell_0,\ell_1,\ell_2\}$. In this setting, we claim that (after being given an ordering of the endpoints of the arcs) we can choose $x_1,x_2,y_1$ and $y_2$ in $k_0 \cdot \ell_0 \cdot m_0 \cdot m_2$ ways. Indeed, we can still pick $x_2,y_2$ in $m_0 \cdot m_2$ ways, whereas we have at most $k_0$ ways to pick $x_1$ and $\ell_0$ ways to pick $y_1$. In this case, we get a contribution of at most

\[
    O(n^{3/2}) \cdot \sum m_0 \ell_0^{-1/2} k_0^{-1/2} \cdot \ell_1^{-3/2} \cdot m_2 k_2^{-3/2} \ell_2^{-3/2}.
\]
Using $O(n^{3/2}) \cdot \ell_1^{-3/2} = O(1)$ we have to evaluate
\[
\left(\sum_{k_0,\ell_0} m_0 \ell_0^{-1/2} k_0^{-1/2} \right) \cdot \left(\sum_{k_2,\ell_2} m_2 k_2^{-3/2} \ell_2^{-3/2}\right),
\]
where the second sum is $O(\log n)$ as before and by a similar argument we find that the first sum is $O(n^2)$, showing that this contribution is $O(n^2 \log n) = o(n^2 \log^2 n)$.
\end{proof}

Lastly, we handle all quadruples that are valid but not good.
\begin{proof}[Proof of Lemma~\ref{L-variancenotgood}]
By Lemma~\ref{L-upperbound} we know that $\Pr[A(\b{x},\b{k},\b{y},\b{l})] = O(k_1^{-3/2} k_2^{-3/2} \ell_1^{-3/2} \ell_2^{-3/2})$. Also, we know by Proposition~\ref{prop-edgeasym} that $\sum_{k_i,\ell_i} g(k_i,\ell_i) k_i^{-3/2} \ell_i^{-3/2} = O(n \log n)$ and by Proposition~\ref{prop-edges} that $\sum_{k_i,\ell_i} n k_i^{-3/2} \ell_i^{-3/2} = o(n \log n)$.

Our goal is to show that given $(k_1,k_2,\ell_1,\ell_2)$ there are at most $O(g(k_1,\ell_1) n + n g(k_2,\ell_2) + n^2)$ quadruples $q \in Q_4$, since then the desired contribution is at most
\[
O\left(\sum_{k_1,\ell_1,k_2,\ell_2} (g(k_1,\ell_1) n + n g(k_2,\ell_2) + n^2) k_1^{-3/2} \ell_1^{-3/2} k_2^{-3/2} \ell_2^{-3/2} \right),
\]
which is the sum of
\begin{align*}
    O\left(\left(\sum_{k_1,\ell_1} g(k_1,\ell_1) k_1^{-3/2} \ell_1^{-3/2}\right)\cdot\left(\sum_{k_2,\ell_2}n k_2^{-3/2} \ell_2^{-3/2}\right)\right) &= O(n \log n) \cdot o(n \log n) = o(n^2 \log^2 n)\\
    O\left(\left(\sum_{k_1,\ell_1}n k_1^{-3/2} \ell_1^{-3/2}\right)\cdot\left(\sum_{k_2,\ell_2}g(k_2,\ell_2) k_2^{-3/2} \ell_2^{-3/2}\right)\right) &= o(n \log n) \cdot O(n \log n) = o(n^2 \log^2 n)\\
    O\left(\left(\sum_{k_1,\ell_1}n k_1^{-3/2} \ell_1^{-3/2}\right)\cdot\left(\sum_{k_2,\ell_2}n k_2^{-3/2} \ell_2^{-3/2}\right)\right) &= o(n \log n) \cdot o(n \log n) = o(n^2 \log^2 n)
\end{align*}
so the total contribution is $o(n^2 \log^2 n)$ as well.

Now, given $(k_1,k_2,\ell_1,\ell_2)$ there are only a few ways in which we can have a valid but not good quadruple.
    \begin{itemize}
        \item $(x_1,k_1,y_1,\ell_1)$ is good, but $(x_2,k_2,y_2,\ell_2)$ is not good. In this case we can pick $(x_1,y_1)$ in at most $g(k_1,\ell_1)$ ways and $(x_2,y_2)$ in $O(n)$ ways, so we are done.
        \item $(x_2,k_2,y_2,\ell_2)$ is good, but $(x_1,k_1,y_1,\ell_1)$ is not good. Similarly to the previous case this will give a bound of $O( n g(k_2,\ell_2))$.
        \item Neither of the $(x_i,k_i,y_i,\ell_i)$ are good. In this case we get a bound of $O(n^2)$ as there are $O(n)$ ways to pick any individual $(x_i,k_i,y_i,\ell_i)$.
        \item Both of the $(x_i,k_i,y_i,\ell_i)$ are good, but the endpoint of one arc of the first four-tuple is adjacent to the endpoint of an arc of the second four-tuple. Note that there are only finitely many possible orderings of the endpoints of the arcs. Given an ordering, there are now at most $g(k_1,\ell_1)$ ways to pick $(x_1,y_1)$, which determines either $x_2$ or $y_2$ since one of $\{x_2,x_2+k_2,y_2,y_2+\ell_2\}$ is adjacent to a now known point, and after that there are at most $2n$ ways to pick the other of $x_2,y_2$, so there are $O(g(k_1,\ell_1) \cdot n)$ possible quadruples in this case. \qedhere
    \end{itemize}
\end{proof}


\begin{thebibliography}{99}


\bibitem{AlonSpencer} N. Alon, and J. Spencer. The Probabilistic Method.
    John Wiley \& Sons, Inc., Hoboken, New Jersey, 2008.

\bibitem{MGT} B. Bollob\'{a}s. Modern graph theory. Vol. 184. Springer Science \& Business Media, 2013.

\bibitem{Catalan-pair} S. Butler, E. Demaine, M. Demaine, R. Graham, A. Hesterberg, J. Ku, J. Lynch and T. Tokieda. Paperclip graphs.

\bibitem{RC} L. Devroye, et al. "Properties of random triangulations and trees." Discrete \& Computational Geometry 22.1 (1999): 105-117.

\bibitem{Gnm} P. Erd\H{o}s, A. R\'{e}nyi. "On the evolution of random graphs." Publ. Math. Inst. Hung. Acad. Sci 5.1 (1960): 17-60.

\bibitem{AC} P. Flajolet, R. Sedgewick. Analytic combinatorics. cambridge University press, 2009.

\bibitem{Gnp} E. Gilbert. "Random graphs." The Annals of Mathematical Statistics 30.4 (1959): 1141-1144.

\bibitem{Nash} N. Nash, Nicholas, D. Gregg. "An output sensitive algorithm for computing a maximum independent set of a circle graph." (2010).

\bibitem{PakCat} I. Pak. "History of Catalan numbers." arXiv preprint arXiv:1408.5711 (2014).

\bibitem{Spinrad} J. Spinrad. "Recognition of circle graphs." Journal of Algorithms 16.2 (1994): 264-282.

\bibitem{StanleyCat} R. Stanley. Catalan numbers. Cambridge University Press, 2015.

\bibitem{Tiskin} A. Tiskin. "Fast distance multiplication of unit-Monge matrices." Algorithmica 71.4 (2015): 859-888.
\end{thebibliography}
\end{document}